\documentclass[letter, 11pt]{article}
\usepackage[T1]{fontenc}
\usepackage{tgpagella}
\usepackage{graphicx,amsmath,amsfonts,amscd,amssymb,bm,url,color,latexsym}
\usepackage{fullpage}
\usepackage[small,bf]{caption}
\usepackage{subcaption}
\usepackage{bbm}
\usepackage{microtype}
\usepackage{algorithm}
\usepackage{algorithmicx}
\usepackage{algpseudocode} 
\usepackage{verbatim}
\usepackage{framed}
\usepackage{comment}
\usepackage{tikz}
\usepackage{fge}
\usepackage{mathtools}
\usepackage{booktabs}
\usepackage{multirow}
\usepackage{multicol}
\usepackage{graphics}
\usepackage{epstopdf}
\usepackage{marvosym}
\usepackage{fancyhdr}
\usepackage{amsrefs}
\usepackage[headsep=6.0mm]{geometry}

\allowdisplaybreaks  

\usepackage{hyperref}
\hypersetup{
    colorlinks=true,%
    citecolor=blue,%
    filecolor=blue,%
    linkcolor=blue,%
    urlcolor=blue
}
\usepackage[toc, page]{appendix}
\usepackage{cleveref}

\newtheorem{theorem}{Theorem}[section]
\newtheorem{lemma}[theorem]{Lemma}
\newtheorem{corollary}[theorem]{Corollary}

\newtheorem{definition}[theorem]{Definition}

\newtheorem{remark}[theorem]{Remark}

\renewcommand{\mathbf}{\boldsymbol}

\numberwithin{equation}{section}

\def \endprf{\hfill {\vrule height6pt width6pt depth0pt}\medskip}
\newenvironment{proof}{\noindent {\bf Proof} }{\endprf\par}

\newcommand{\tensor}[1]{\bm{\mathcal{#1}}}

\title{Provable Low-Rank Tensor-Train Approximations in the Inverse of Large-Scale Structured Matrices}
\author{Chuanfu Xiao$^1$, Kejun Tang$^2$, Zhitao Zhu$^3$ \\
\vspace{2mm}
\texttt{\{chuanfuxiao, zhuzt\}@pku.edu.cn, $\textrm{\Letter}$tangkejun@suat-sz.edu.cn} \\
$^1$PKU-Changsha Institute of Computing and Digital Economy, Changsha, P.R. China\\
$^2$Faculty of Computility Microelectronics, \\Shenzhen University of Advanced Technology, Shenzhen, P.R. China\\
$^3$School of Mathematical Science, Peking University, Beijing, P.R. China
}

\begin{document}
\fancyhead{} 
\fancyhead[C]{\textbf{Provable Low-Rank Tensor-Train Approximations}}

\maketitle

\vspace{-0.3in}
\begin{abstract}
This paper studies the low-rank property of the inverse of a class of large-scale structured matrices in the tensor-train (TT) format, which is typically discretized from differential operators. 
An interesting question that we are concerned with is: Does the inverse of the large-scale structured matrix still admit the low-rank TT representation with guaranteed accuracy? In this paper, we provide a computationally verifiable sufficient condition such that the inverse matrix can be well approximated in a low-rank TT format. It not only answers what kind of structured matrix whose inverse has the low-rank TT representation but also motivates us to develop an efficient TT-based method to compute the inverse matrix. Furthermore, we prove that the inverse matrix indeed has the low-rank tensor format for a class of large-scale structured matrices induced by differential operators involved in several PDEs, such as the Poisson, Boltzmann, and Fokker-Planck equations. Thus, the proposed algorithm is suitable for solving these PDEs with massive degrees of freedom. Numerical results on the Poisson, Boltzmann, and Fokker-Planck equations validate the correctness of our theory and the advantages of our methodology.
\end{abstract}






\section{Introduction}\label{sec:intro}

Suppose we are given a large matrix $\bm{L}\in\mathbb{C}^{\prod\limits_{k=1}^dn_k\times \prod\limits_{k=1}^dn_k}$ with \emph{full rank}, and know that it has the following form
\begin{equation}\label{eq:coefficient matrix}
  \bm{L} = \sum\limits_{k=1}^d\bm{M}^{(1)}\otimes\cdots\otimes\bm{M}^{(k-1)}\otimes\bm{S}^{(k)}\otimes\bm{M}^{(k+1)}\otimes\cdots\otimes\bm{M}^{(d)},
\end{equation} 
where $\otimes$ denotes the Kronecker product, and $\bm{M}^{(k)},\ \bm{S}^{(k)}\in\mathbb{C}^{n_k\times n_k}$ are joint diagonalizable matrices for all $k=1,2,\ldots,d$, i.e., there exists two invertible matrices $\bm{U}_k$ and $\bm{V}_k$ such that 
$\bm{U}_k\bm{M}^{(k)}\bm{V}_k$ and $\bm{U}_k\bm{S}^{(k)}\bm{V}_k$
both are diagonal matrices. Our goal is to accurately and efficiently compute the inverse of $\bm{L}$, which is a fundamental task in the field of scientific computing and plays an important role in various applications \cite{golub2013matrix,kailath1980linear,kay2013fundamentals,murphy2012machine,benzi2002preconditioning}. From the perspective of matrices, it is challenging to directly compute the inverse of $\bm{L}$ for a large $d$ due to the curse of dimensionality. To this end, we resort to the perspective of tensor representations.
As pointed out in \cite{kazeev2012low}, matrices of the form \eqref{eq:coefficient matrix} can be represented as a tensor-train (TT) format with TT ranks $(1,2,\ldots,2,1)$.
The low-rank TT representation provides an excellent opportunity to achieve our goal. However, a crucial question for using low-rank TT representations is whether the inverse of $\bm{L}$ is indeed low-rank in the TT format. 

A \emph{provable} and \emph{scalable} solution to the above question is the theoretical foundation for iterative tensor-based matrix inversion algorithms. In terms of algorithms, Oseledets and Dolgov proposed as early as 2012 to combine the low-rank TT representation of $\bm{L}$ and use iterative methods such as density matrix renormalization group (DMRG) and Newton's iteration to compute its inverse \cite{oseledets2012solution}. Although such an iterative algorithm based on the low-rank TT format is efficient and widely used, there is a gap between the theory and the algorithm. That is, can the inverse of $\bm{L}$ also be represented by the low-rank TT format with guaranteed accuracy? It is the key to ensure that TT-based methods can successfully provide the inverse matrix $\bm{L}^{-1}$ with guaranteed accuracy at low costs. In this work, we conduct an in-depth analysis of whether the inverse matrix $\bm{L}^{-1}$ has a low-rank structure under the TT format. More precisely, we will answer the following two questions: (1) Does the inverse matrix $\bm{L}^{-1}$ still have a low-rank TT representation with guaranteed accuracy? (2) If so, can we design a more efficient algorithm to compute the inverse matrix $\bm{L}^{-1}$? 

\paragraph{Applications.} The large-scale matrices defined in \eqref{eq:coefficient matrix} are typically derived from PDE operators discretized by the finite difference method (FDM) \cite{smith1985numerical,strikwerda2004finite} or the finite element method (FEM) \cite{zienkiewicz2005finite,reddy2019introduction}.
For example, considering the $d$-dimensional Poisson equation defined in a hypercube with the Dirichlet boundary condition, matrices constructed on tensor-product uniform meshes by FDM and FEM are
\begin{equation*}
    \sum\limits_{k=1}^d\bm{I}\otimes\cdots\otimes\bm{I}\otimes\bm{\Delta}\otimes\bm{I}\otimes\cdots\otimes\bm{I}\quad \text{and}\quad \sum\limits_{k=1}^d\bm{M}^{(1)}\otimes\cdots\otimes\bm{M}^{(k-1)}\otimes\bm{\Delta}\otimes\bm{M}^{(k+1)}\otimes\cdots\otimes\bm{M}^{(d)},
\end{equation*}
where $\bm{I}$ is the $n\times n$ identity matrix, and 
\begin{equation*}
    \bm{\Delta} = \frac{1}{h^2}\left[\begin{array}{ccccc}
       2  & -1 &  &  &  \\
        -1 & 2 & \ddots &  &  \\
         & \ddots & \ddots & -1 &  \\
         &  & -1 & 2 & -1 \\
         &  &  & -1 & 2 \\
    \end{array}\right]\quad \text{and}\quad \bm{M}^{(k)} = \left[\begin{array}{ccccc}
       4  & 1 &  &  &  \\
        1 & 4 & \ddots &  &  \\
         & \ddots & \ddots & 1 &  \\
         &  & 1 & 4 & 1 \\
         &  &  & 1 & 4 \\
    \end{array}\right]
\end{equation*}
represent the one-dimensional negative discretized Laplace operator and mass matrix, respectively. Besides, the discretized differential operators involved in other PDEs such as the Boltzmann and Fokker-Planck equations, also have the form \eqref{eq:coefficient matrix} (see \Cref{sec:application} for more details).

\paragraph{Related Work.} In the past few decades, there have been some valuable works studying the low-rank property of the inverse of $\bm{L}$ discretized from PDEs in different tensor formats. For example, Kazeev and Khoromskij considered the one-dimensional discretized Laplace operator and explicitly gave the low-rank quantized TT (QTT) representation of its inverse \cite{kazeev2012low}, but it cannot be generalized to the $d$-dimensional case ($d > 1$). Oseledets estimated the TT rank of the inverse of the two-dimensional discretized Laplace operator is bounded by $\mathcal{O}(\log n)$ through extensive numerical experiments \cite{oseledets2010approximation}. For more general discretized elliptic operators, many researchers have analyzed the low-rank structure of their inverse in CANDECOMP/PARAFAC (CP) format, see \cite{beylkin2002numerical,bebendorf2003existence,grasedyck2004existence,gavrilyuk2005hierarchical,hackbusch2006low,hackbusch2007tensor,hackbusch2008tensor,khoromskij2009tensor,kazeev2012low}.
Specifically, they constructed a low-rank CP approximation of the inverse of the discretized elliptic operator via a separable approximation of the kernel function corresponding to the elliptic operator. However, giving an approximate separable expansion of the kernel function is computationally expensive\cite{hackbusch2008tensor,khoromskij2009tensor}, and only is suitable for elliptic operators. In addition, Dahmen \emph{et al.} also considered a pioneering method for constructing the low-rank CP approximation of the inverse of elliptic operators by its eigendecomposition and then approximating the Hadamard inverse of eigenvalues \cite{dahmen2016tensor}.

\subsection{Summary of Contributions}

In this paper, we prove that the inverse of matrices with structure \eqref{eq:coefficient matrix} can be well approximated under a broad condition in the low-rank TT format, and further derive bounds on its TT rank for a predetermined accuracy requirement. More importantly, the given condition to guarantee the low-rank TT format of the inverse matrix is computationally verifiable.
Furthermore, we also prove that matrices discretized from several PDEs, including the Poisson, Boltzmann, and Fokker-Planck equations not only have the form \eqref{eq:coefficient matrix} but also satisfy the given sufficient condition. 
To efficiently compute the inverse of matrices of the form \eqref{eq:coefficient matrix}, we develop a novel TT-based matrix inversion method that converts the matrix inversion operation into the Hadamard inverse of the corresponding tensor in TT format through diagonalization. Numerical experiments validate the correctness of our theory and demonstrate the advantages of the TT-based matrix inversion method for solving the Poisson, Boltzmann, and Fokker-Planck equations.

Compared with literature \cite{dahmen2016tensor}, there are two key differences in our work. First, we provide a condition that can be \emph{computationally verified} by the structured matrix itself to determine whether its inverse enjoys a low-rank TT structure. Second, the theory we established is not limited to discretized elliptic operators but also applies to matrices discretized from advection equations, such as the Boltzmann-BGK equation.

\subsection{Organization}
The rest of this paper is organized as follows. In \Cref{sec:pre}, we briefly introduce the TT format and some operations on TT formats required in this paper. \Cref{sec:algorithm} presents the TT-based matrix inversion method for computing the inverse of matrices of the form \eqref{eq:coefficient matrix}.  In \Cref{sec:analysis}, we give a sufficient condition such that the inverse matrix can be well approximated in a low-rank TT format, and then illustrate that it is computationally verifiable. \Cref{sec:application} shows that matrices discretized from the Poisson, Boltzmann, and Fokker-Planck equations not only have the form \eqref{eq:coefficient matrix}, but also satisfy the condition given in \Cref{sec:analysis}, which means that the proposed TT-based matrix inversion method can be directly used to solve these PDEs, and the corresponding numerical results are reported in \Cref{sec:example}. This paper is concluded in \Cref{sec:conclusion}.

\section{Preliminarlies}\label{sec:pre}

In this paper, we use the bold lowercase letter to represent a vector, e.g., $\bm{a}\in\mathbb{C}^{n}$, the bold uppercase letter to represent a matrix, e.g., $\bm{A}\in\mathbb{C}^{m\times n}$, and the bold script letter to represent a $d$th-order tensor, e.g., $\tensor{A}\in\mathbb{C}^{n_1\times n_2\times\cdots\times n_d}$. For convenience, other symbols and notations used in this paper are summarized in \Cref{table:symbol}.

\begin{table}[!htb]
    \centering
    \caption{Symbols and notations used in this paper.}
    \begin{tabular}{c|c}
    \toprule
        Symbol & Description \\
        \midrule
        
       $\tensor{A}({i_1,i_2,\ldots,i_d})$ & $(i_1,i_2,\ldots,i_d)$th element of $\tensor{A}$ \\
       $\bm{A}_{(\mathcal{I}_k)}$ & Matricization of $\tensor{A}$ along the index set $\mathcal{I}_k=\{1,\ldots,k\}$\\
        $\|\cdot\|_F$ & Frobenius norm\\
        $\times_k$ & Mode-$k$ product of a tensor and a matrix \\
        $\otimes$ & Kronecker product\\
        $\odot$ & Hadamard product \\
        $\bm{I}_{n}$ & An $n\times n$ identity matrix\\
        $\kappa(\bm{A})$ & Condition number of $\bm{A}$\\
        $\lceil\cdot\rceil$ & Ceiling function\\
        \bottomrule
    \end{tabular}
    \label{table:symbol}
\end{table}

\subsection{Basic tensor operations}

Let $\tensor{A}\in\mathbb{C}^{n_1\times n_2\times \cdots \times n_d}$ be a complex $d$th-order tensor, 
the conjugate of $\tensor{A}$ is denoted as $\overline{\tensor{A}}$ with entries $\overline{\tensor{A}}({i_1,i_2,\ldots,i_d})=\overline{\tensor{A}({i_1,i_2,\ldots,i_d})}$. The Frobenius norm of $\tensor{A}$ is defined as 
\begin{equation*}
    \|\tensor{A}\|_{F} = \sqrt{\sum\limits_{i_1,i_2,\ldots,i_d}|\tensor{A}({i_1,i_2,\ldots,i_d})|^2},
\end{equation*}
and it is easy to know that the Frobenius norm of $\overline{\tensor{A}}$ is equivalent to $\tensor{A}$, i.e., $\|\overline{\tensor{A}}\|_F = \|\tensor{A}\|_F$.

Matricization is an operation that rearranges a tensor into a matrix according to a certain bijection. A typical mapping for matricization is as follows. Let $\mathcal{I}_k=\{1,\ldots,k\}$ ($k\in\{1, 2, \ldots, d-1\}$) be an index set, then $\bm{A}_{\left(\mathcal{I}_k\right)}\in\mathbb{C}^{\sum\limits_{s=1}^kn_s\times\sum\limits_{s=k+1}^dn_s}$ is the matricization of $\tensor{A}$ along the index set $\mathcal{I}_k$, and satisfies the following mapping relationship:
\begin{equation*}
    \bm{A}_{\left(\mathcal{I}_k\right)}({i,j}) = \tensor{A}({i_1,i_2,\ldots,i_d})\quad \text{with}\quad\left\{\begin{array}{c}
         i =  i_k + (i_{k-1}-1)n_k + \cdots + (i_1-1)n_2\cdots n_k \\
        j = i_d + (i_{d-1}-1)n_d + \cdots + (i_{k+1}-1)n_{k+2}\cdots n_d
    \end{array}\right..
\end{equation*}

The mode-$k$ product of $\tensor{A}$ and a matrix $\bm{U}\in\mathbb{C}^{m_k\times n_k}$ is also a $d$th-order tensor, which is denoted as $\tensor{B} = \tensor{A}\times_{k}\bm{U}\in\mathbb{C}^{n_1\times\cdots\times n_{k-1}\times m_{k}\times n_{k+1}\times\cdots\times n_d}$. Elementwisely, we have
\begin{equation*}
    \tensor{B}({i_1,\ldots,j,\ldots,i_d}) = \sum\limits_{i_k=1}^{n_k}\tensor{A}({i_1,\ldots,i_k,\ldots,i_d})\bm{U}({j,i_k}).
\end{equation*}
Particularly, the product of $\tensor{A}$ and a series of matrices $\left\{\bm{U}^{(k_l)}\in\mathbb{C}^{m_{k_l}\times n_{k_l}}:l=1,\ldots,L\right\}$ along the corresponding modes is called tensor-times-matrix chain (TTMc), and denoted as 
\[
\tensor{A}\times_{k_1}\bm{U}^{(k_1)}\times_{k_2}\bm{U}^{(k_2)}\cdots\times_{k_L}\bm{U}^{(k_L)}.
\]


Let $\tensor{B}\in\mathbb{C}^{m_1\times m_2\times \cdots\times m_d}$ be a $d$th-order tensor, the Kronecker product of $\tensor{A}$ and $\tensor{B}$ yields a $d$th-order tensor $\tensor{C}\in\mathbb{C}^{n_1m_1\times n_2m_2\times\cdots\times n_dm_d}$ with entries 
\begin{equation*}
    \tensor{C}({s_1,s_2,\ldots,s_d}) = \tensor{A}({i_1,i_2,\ldots,i_d})\tensor{B}({j_1,j_2,\ldots,j_d}),
\end{equation*}
where $s_k = i_k + (j_k-1)n_k$ for all $k=1,2,\ldots,d$.

The Hadamard product of $\tensor{A}$ and $\tensor{B}$ requires that $\tensor{B}$ be of the same size as $\tensor{A}$, which also yields a $d$th-order tensor $\tensor{C}$ with entries 
\begin{equation*}
    \tensor{C}({i_1,i_2,\ldots,i_d}) = \tensor{A}({i_1,i_2,\ldots,i_d})\tensor{B}({i_1,i_2,\ldots,i_d}).
\end{equation*}
For convenience, the Hadamard product of $\tensor{A}$ and itself is denoted as $\tensor{A}^{\odot2}$. The Hadamard inverse of $\tensor{A}$ is denoted as $\tensor{A}^{\odot-1}$ with entries
\begin{equation*}
    \tensor{A}^{\odot-1}({i_1,i_2,\ldots,i_d}) = \frac{1}{\tensor{A}({i_1,i_2,\ldots,i_d})}.
\end{equation*}

\subsection{TT decomposition}

The TT decomposition \cite{oseledets2011tensor} decomposes the $d$th-order tensor $\tensor{A}$ into $d$ core tensors with order no more than three, 
which is formally defined as 
\begin{equation}\label{eq:tt}
    \tensor{A}({i_1,i_2,\ldots,i_d}) = \tensor{A}^{(1)}({i_1,:})\tensor{A}^{(2)}({:,i_2,:})\cdots\tensor{A}^{(d)}({:,i_d}),
\end{equation}
where $\left\{\tensor{A}^{(k)}\in\mathbb{C}^{r_{k-1}\times n_k\times r_k}:k=1,2,\ldots,d\right\}$ are the TT cores of $\tensor{A}$, and $(r_0,r_1,\ldots,r_d)$ is called the TT ranks that satisfies $r_0=r_{d}=1$. 
It is easy to see that the TT representation is able to reduce the memory cost of $\tensor{A}$ from $\mathcal{O}\left(\prod\limits_{k=1}^dn_k\right)$ to $\mathcal{O}\left(\sum\limits_{k=1}^dn_kr_{k-1}r_k\right)$, which only grows linearly with the order $d$ when its TT ranks are small enough.

For convenience, if the $d$th-order tensor $\tensor{A}$ appears in the TT format \eqref{eq:tt}, then we call it the $d$th-order TT tensor. We borrow the notation defined in \cite{lee2014fundamental} and denote the $d$th-order TT tensor $\tensor{A}$ as
\begin{equation}\label{eq:tt-block}
    \tensor{A} = \tensor{A}^{(1)}\times^1\tensor{A}^{(2)}\times^1\cdots\times^1\tensor{A}^{(d)}.
\end{equation}

\subsection{Operations on TT formats}
Another important advantage of TT representation is that basic linear algebra operations can be implemented into corresponding operations on TT cores, allowing the potential to reduce time and memory costs. For example, the conjugate of $\tensor{A}$ can be easily implemented by sequentially taking the conjugate of TT cores, i.e., 
\begin{equation*}
\overline{\tensor{A}} = \overline{\tensor{A}^{(1)}}\times^1\overline{\tensor{A}^{(2)}}\times^1\cdots\times^1\overline{\tensor{A}^{(d)}}.
\end{equation*}
The mode-$k$ product of $\tensor{A}$ and $\bm{U}$ only requires the mode-$2$ product of the $k$th TT core $\tensor{A}^{(k)}$ and $\bm{U}$. For the Kronecker product, we only need to sequentially perform the Kronecker product operation on TT cores. The addition and Hadamard product of two $d$th-order TT tensors can be implemented by the concatenation and Kronecker product operations on their TT cores along the mode-1 and -3, respectively. More details about tensor operations on TT formats can be found in \cite{lee2014fundamental,cichocki2016tensor,cichocki2017tensor}.

We remark that, most basic linear algebra operations will cause the TT rank of tensors to increase. For example, the addition operation increases the TT rank to $2r$, and the Kroneker and  Hadamard products increase the TT rank to $r^2$. To suppress the growth of TT rank, the recompression operation is required for rank reduction. It is actually equivalent to finding the best low TT rank approximation of TT tensors and can be solved by the classical TT-rounding algorithm and its high-performance implementations \cite{oseledets2011tensor,daas2022parallel,al2022parallel,al2023randomized,shi2023parallel}.
\section{TT-based matrix inversion method}\label{sec:algorithm}

In this section, we introduce a novel TT-based method for computing the inverse of the structured matrix $\bm{L}\in\mathbb{C}^{\prod\limits_{k=1}^dn_k\times \prod\limits_{k=1}^dn_k}$ of the form \eqref{eq:coefficient matrix}. Specifically, we convert the matrix inversion operation into the Hadamard inverse of a corresponding $d$th-order TT tensor via diagonalization, and then use iterative methods to solve it. Compared with the approach in \cite{oseledets2012solution}, the operations involved in our proposed method are all based on $d$th-order TT tensors instead of $2d$th-order TT tensors, thus it has higher computational efficiency. More importantly, we can give an in-depth analysis of the low-rank structure of the inverse matrix $\bm{L}^{-1}$ in the TT format following our proposed method, which will be presented in \Cref{sec:analysis}.

\subsection{Diagonalization}\label{subsec:diag}
Since the structured matrix $\bm{L}$ is of the form \eqref{eq:coefficient matrix} with full rank, there are two series of invertible matrices 
$\left\{\bm{U}_{k}\in\mathbb{C}^{n_k\times n_k}:k=1,2,\ldots,d\right\}$ and $\left\{\bm{V}_{k}\in\mathbb{C}^{n_k\times n_k}:k=1,2,\ldots,d\right\}$ such that 
\begin{equation*}\label{eq:diag}
    \bm{U}_{k}\bm{S}^{(k)}\bm{V}_{k} = \bm{D}_k\quad \text{and}\quad \bm{U}_k\bm{M}^{(k)}\bm{V}_k = \bm{\Lambda}_{k},
\end{equation*}
where 
\[
\bm{D}_k=\text{diag}\left(\mu_{1}^{(k)},\mu^{(k)}_{2},\ldots,\mu_{n_k}^{(k)}\right)\quad \text{and}\quad \bm{\Lambda}_{k}=\text{diag}\left(\lambda_{1}^{(k)},\lambda^{(k)}_{2},\ldots,\lambda_{{n_k}}^{(k)}\right)
\]
are two diagonal matrices with non-zero entries.
With the invertible transformations $\bm{U}_1\otimes\bm{U}_2\otimes\cdots\otimes\bm{U}_d$ and $\bm{V}_1\otimes\bm{V}_2\otimes\cdots\otimes\bm{V}_d$, the matrix $\bm{L}$ can be transformed into a diagonal matrix as follows
\begin{equation}\label{eq:diag_d}
\sum\limits_{k=1}^d\bm{\Lambda}_{1}\otimes\cdots\otimes\bm{\Lambda}_{k-1}\otimes\bm{D}_k\otimes\bm{\Lambda}_{k+1}\otimes\cdots\otimes\bm{\Lambda}_{d}.
\end{equation}
Let $\tensor{L}$ be the tensor form of $\bm{L}$ after diagonalization, which is a $d$th-order tensor in $\mathbb{C}^{n_1\times n_2\times\cdots\times n_d}$. Then computing the inverse matrix $\bm{L}^{-1}$ is equivalent to 
computing the Hadamard inverse of the corresponding $d$th-order tensor $\tensor{L}\in\mathbb{C}^{n_1\times n_2\times\cdots\times n_d}$ with entries
\begin{equation}\label{eq_TTtensor_after_diag}
\begin{split}
    \tensor{L}({j_1,j_2,\ldots,j_d}) & = \sum\limits_{k=1}^d\lambda_{j_1}^{(1)}\cdots\lambda^{(k-1)}_{j_{k-1}}\mu^{(k)}_{j_k}\lambda_{j_{k+1}}^{(k+1)}\cdots\lambda^{(d)}_{j_d},
    \end{split}
\end{equation}
that is, the inverse of the diagonal matrix in \eqref{eq:diag_d}. 

\subsection{Low-rank TT representation}
As referred in \cite{oseledets2011tensor,kazeev2012low}, we can easily give the rank-$(1,2,\ldots,2,1)$ TT representation of the tensor $\tensor{L}$ without any additional computational cost. Such a representation transforms a lower-order tensor (matrices are second-order tensors) to a higher-order one, which allows us to exploit the low-dimensional structures and reduce the computational cost. Although $\bm{L}$ is indeed full rank in the matrix sense, its tensor form has the low-rank property. This order promotion technique is crucial for our algorithm and analysis. Let $\bm{\mu}^{(k)},\ \bm{\lambda}^{(k)}\in\mathbb{C}^{n_k}$ be two row vectors consisting of $\left\{\mu_{1}^{(k)},\mu_{2}^{(k)},\ldots,\mu_{n_k}^{(k)}\right\}$ and $\left\{\lambda_{1}^{(k)},\lambda_{2}^{(k)},\ldots,\lambda_{n_k}^{(k)}\right\}$, the TT representation of $\tensor{L}$ is specifically expressed as follows:
\begin{equation}\label{eq:tt-coefficient}
    \tensor{L} = \tensor{L}^{(1)}\times^1\tensor{L}^{(2)}\times^1\cdots\times^{1}\tensor{L}^{(d)},
\end{equation}
where the TT cores satisfy
\begin{equation*}
\tensor{L}^{(1)}\in\mathbb{C}^{1\times n_1\times 2}:=\left\{\begin{array}{c}
    \tensor{L}^{(1)}({:,:,1})=\bm{\lambda}^{(1)}  \\
    \tensor{L}^{(1)}({:,:,2}) = \bm{\mu}^{(1)} 
\end{array}\right.,\quad \tensor{L}^{(d)}\in\mathbb{C}^{2\times n_d\times 1}:=\left\{\begin{array}{c}
    \tensor{L}^{(d)}({1,:,:})=\bm{\mu}^{(d)}  \\
    \tensor{L}^{(d)}({2,:,:}) = \bm{\lambda}^{(d)}
\end{array}\right.,
\end{equation*}
and
\begin{equation*}
\tensor{L}^{(k)}\in\mathbb{C}^{2\times n_k\times 2}:=\left\{\begin{array}{c}
    \tensor{L}^{(k)}({1,:,1})=\bm{\lambda}^{(k)}  \\
    \tensor{L}^{(k)}({2,:,1}) = \bm{0} \\
    \tensor{L}^{(k)}({1,:,2}) = \bm{\mu}^{(k)} \\
    \tensor{L}^{(k)}({2,:,2}) = \bm{\lambda}^{(k)} \\
\end{array}\right.
\end{equation*}
for all $k=2,\ldots,d-1$.


\subsection{Hadamard inverse of TT tensors}
Now the matrix inversion problem turns into how to compute the Hadamard inverse of $\tensor{L}$ since it is a tensor representation of the diagonalization of $\bm{L}$. Directly computing the Hadamard inverse of $\tensor{L}$ requires $\mathcal{O}\left(\prod\limits_{k=1}^dn_k\right)$ time and memory costs, which grows exponentially with the dimension $d$, and still suffers from the curse of dimensionality. To remedy this issue, we consider the following equivalent nonlinear equation:
\begin{equation}\label{eq:hadamardinv_eq}
    \tensor{X}^{\odot-1} = \tensor{L},
\end{equation}
where $\tensor{X}$ denotes the Hadamard inverse of $\tensor{L}$.
With the low-rank TT representation of $\tensor{L}$, \cite{espig2020iterative} pointed out that the nonlinear equation \eqref{eq:hadamardinv_eq} can be efficiently solved by Newton's iteration for the case where the tensor is real. We now show that the method presented in \cite{espig2020iterative} is also suitable for complex tensors.
For the complex TT tensor $\tensor{L}$, we can write $\tensor{L}$ as the sum of its real and imaginary parts: $\tensor{L}=\tensor{L}_1+i\tensor{L}_2$, where
\begin{equation}\label{eq:real_img}
    \tensor{L}_1 = \frac{\tensor{L} + \overline{\tensor{L}}}{2}\quad \text{and}\quad \tensor{L}_2 = \frac{\tensor{L} - \overline{\tensor{L}}}{2i}.
\end{equation}
Since the Hadamard inversion is an element-wise operator, the Hadamard inverse of $\tensor{L}$ is given by 
\begin{equation}\label{eq:hinverse}
    \tensor{L}^{\odot-1} = \left(\tensor{L}_1^{\odot2}+\tensor{L}_2^{\odot2}\right)^{\odot-1}\odot(\tensor{L}_1-i\tensor{L}_2),
\end{equation}
where $\tensor{L}_1^{\odot2}+\tensor{L}_2^{\odot2}$ is a real $d$th-order TT tensor with TT ranks at most $(1,32,\ldots,32,1)$.
There are two choices to obtain the Hadamard inverse of the complex tensor $\tensor{L}$. The first is to use Newton's iteration to compute the Hadamard inverse of the TT tensor $\tensor{L}_1^{\odot2} + \tensor{L}_2^{\odot2}$. Then the low-rank TT approximation of $\tensor{L}^{\odot-1}$ can be constructed  by \eqref{eq:real_img} and \eqref{eq:hinverse}. Obviously, we need to explicitly compute $\left(\tensor{L}_1^{\odot2}+\tensor{L}_2^{\odot2}\right)^{\odot-1}$ if we adopt such a choice. The second is also derived from Newton's iteration for computing the Hadamard inverse of $\tensor{L}_1^{\odot2}+\tensor{L}_2^{\odot2}$, its corresponding iterative process is as follows:
\begin{equation}\label{eq:newton}
    \tensor{X}\quad\leftarrow\quad\tensor{X} - \tensor{X}\odot\left(\left(\tensor{L}_1^{\odot2}+\tensor{L}_2^{\odot2}\right)\odot\tensor{X}-\tensor{E}\right),
\end{equation}
where $\tensor{E}\in\mathbb{R}^{n_1\times n_2\times \cdots\times n_d}$ is a $d$th-order tensor whose elements are all 1. It is easy to verify that $\tensor{E}$ can be represented by a rank-one TT tensor. By multiplying the left and right sides of \eqref{eq:newton} $\tensor{L}_1$ and $i\tensor{L}_2$ respectively, then we have
\begin{equation}\label{eq:newton-1}
\left\{\begin{array}{c}
      \tensor{X}\odot\tensor{L}_1\quad\leftarrow\quad\tensor{X}\odot\tensor{L}_1 - \tensor{X}\odot\tensor{L}_1\odot\left(\left(\tensor{L}_1^{\odot2}+\tensor{L}_2^{\odot2}\right)\odot\tensor{X}-\tensor{E}\right) \\
      \tensor{X}\odot i\tensor{L}_2\quad\leftarrow\quad \tensor{X}\odot i\tensor{L}_2 - \tensor{X}\odot i\tensor{L}_2\odot\left(\left(\tensor{L}_1^{\odot2}+\tensor{L}_2^{\odot2}\right)\odot\tensor{X}-\tensor{E}\right)
\end{array}\right..
\end{equation}
From \eqref{eq:newton-1}, we can further obtain 
\begin{equation}\label{eq:newton-2}
    \begin{array}{ccc}
        \tensor{X}\odot\left(\tensor{L}_1-i\tensor{L}_2\right) & \leftarrow &  \tensor{X}\odot\left(\tensor{L}_1 - i\tensor{L}_2\right) - \tensor{X}\odot\left(\tensor{L}_1-i\tensor{L}_2\right)\odot\left(\left(\tensor{L}_1^{\odot2}+\tensor{L}_2^{\odot2}\right)\odot\tensor{X}-\tensor{E}\right)\\
         & = &  \tensor{X}\odot\left(\tensor{L}_1 - i\tensor{L}_2\right) - \tensor{X}\odot\left(\tensor{L}_1-i\tensor{L}_2\right)\odot\left(\tensor{L}\odot\tensor{X}\odot\left(\tensor{L}_1-i\tensor{L}_2\right)-\tensor{E}\right). \\
    \end{array}
\end{equation}
This implies that Newton's iteration can be used directly to compute the Hadamard inverse of complex TT tensors, which is more efficient because it does not need to explicitly compute the tensor $\left(\tensor{L}_1^{\odot2}+\tensor{L}_2^{\odot2}\right)^{\odot-1}$. \Cref{algo:newton} describes the computational procedure of Newton's iteration for computing the Hadamard inverse of TT tensors.

\begin{algorithm}[H]
\caption{Newton's iteration for computing the Hadamard inverse of TT tensors.} 
	\label{algo:newton}
	\begin{algorithmic}[1]
		\Require
  Tensor $\bm{\mathcal{L}}\in\mathbb{C}^{n_1\times n_2\times\cdots\times n_d}$ with TT format, initial value $\tensor{X}_0$, and rounding error $\varepsilon$.
  
		\Ensure The Hadamard inverse of $\tensor{L}$: $\tensor{X}_{k}\in\mathbb{C}^{n_1\times n_2\times\cdots\times n_d}$.
  \State $k\quad \leftarrow\quad 0$
  \While{\text{not convergence}}
  \State Compute the residual: $\tensor{R}_{k}\quad\leftarrow\quad\tensor{L}\odot\tensor{X}_{k}-\tensor{E}$. 
  \State Update: $\tensor{X}_{k+1}\quad \leftarrow\quad \tensor{X}_{k} - \tensor{X}_{k}\odot\tensor{R}_{k}$.
  \State Truncate: $\tensor{X}_{k+1}\quad \leftarrow\quad \texttt{TT-rounding}(\tensor{X}_{k+1},\varepsilon)$.
  \State $k\quad \leftarrow\quad k+1$
  \EndWhile
	\end{algorithmic}
\end{algorithm}

\begin{remark}\label{rmk:hadamard-inverse}
Another way to compute the Hadamard inverse of $\tensor{L}$ is to use the gradient descent method to solve the corresponding linear system:
\begin{equation*}
    \tensor{L}\odot\tensor{X} = \tensor{E}.
\end{equation*}
Specifically, the iterative process is as follows,
\begin{equation*}
    \tensor{X}_{k+1}\quad \leftarrow\quad \tensor{X}_{k} - \alpha\tensor{L}\odot(\tensor{L}\odot\tensor{X}_k-\tensor{E}),
\end{equation*}
where $\alpha>0$ is the iteration step size. Although the gradient descent method is linear convergent, it avoids calculating the Hadamard product of $\tensor{X}_k$ and itself, i.e.,  $\tensor{X}_k\odot\tensor{X}_k$. Therefore, compared with \Cref{algo:newton}, the TT rank of the intermediate tensor $\tensor{X}_k$ is smaller during the iterations of gradient descent, which implies that less memory costs are required. Typically, we can use the gradient descent method to provide a good initial guess for Newton's iteration in practice.
\end{remark}

\subsection{Explicit TT approximation of the invese matrix}\label{sec:step4}
Once we have the Hadamard inverse of $\tensor{L}$ in hand, we can explicitly construct the inverse matrix $\bm{L}^{-1}$ with some additional steps. 
To achieve this, we first need an $\texttt{Expanding}$ operator is defined as follows:
\begin{equation}\label{eq:expand}
    \texttt{Expanding}:\tensor{X}\in\mathbb{C}^{n_1\times n_2\times\cdots\times n_d}\quad \rightarrow\quad \hat{\tensor{L}} = \left(\tensor{X}\otimes \tensor{E}\right)\odot\tensor{I}\in\mathbb{C}^{(n_1\times n_1)\times (n_2\times n_2)\times\cdots\times (n_d\times n_d)},
\end{equation}
where
\begin{equation*}
    \tensor{I} = \bm{I}_{n_1}\otimes\bm{I}_{n_2}\otimes\cdots\otimes\bm{I}_{n_d}\in\mathbb{R}^{(n_1\times n_1)\times (n_2\times n_2)\times\cdots\times (n_d\times n_d)}
\end{equation*}
is the TT representation of the $\prod\limits_{k=1}^dn_k\times\prod\limits_{k=1}^dn_k$ identity matrix. The \texttt{Expanding} operator provides an efficient way to convert the Hadamard inverse of $\tensor{L}$ to the inverse of $\bm{L}$ itself.
With the $\texttt{Expanding}$ operator, it is easy to obtain the TT approximation of $\bm{L}^{-1}$ by the TTMc operation of the tensor $\texttt{Expanding}(\tensor{L}^{\odot-1})$ and matrices $\left\{\bm{U}_k^{-1},\bm{V}_k^{-1}:k=1,2,\ldots,d\right\}$. More specifically, the inverse of $\bm{L}$ is given by 
\begin{equation}\label{eq_final_inverse}
    \begin{gathered}
          \bm{L}^{-1} = \texttt{Expanding}\left(\tensor{L}^{\odot-1}\right)\times_1\bm{U}_1^{-1}\times_2\bm{V}_{1}^{-1}\cdots\times_{2k-1}\bm{U}_{k}^{-1}\times_{2k}\bm{V}_{k}^{-1}\cdots\times_{2d-1}\bm{U}_d^{-1}\times_{2d}\bm{V}_d^{-1}.
    \end{gathered}
\end{equation}
Since both $\tensor{E}$ and $\tensor{I}$ are rank-one TT tensors, the defined $\texttt{Expanding}$ operator will not lead to an increase in TT rank. In other words, the TT rank of the tensor corresponding to the inverse matrix $\bm{L}^{-1}$ is equal to $\tensor{L}^{\odot-1}$.

\subsection{Overall computational procedure of the TT-based matrix inversion method}

Combining the above four steps, we can give the overall computational procedure of the proposed TT-based matrix inversion method, which is summarized in \Cref{algo:solver}.

\begin{algorithm}[H]
\caption{TT-based matrix inversion method.}
\label{algo:solver}
\begin{algorithmic}[1]
\Require 
A structured matrix $\bm{L}\in\mathbb{C}^{\prod\limits_{k=1}^dn_k\times\prod\limits_{k=1}^dn_k}$ of the form \eqref{eq:coefficient matrix}: matrices $\{\bm{S}^{(k)},\bm{M}^{(k)}\in\mathbb{C}^{n_k\times n_k}:k=1,2,\ldots,d\}$.
\Ensure Low-rank approximation of the inverse matrix $\bm{L}^{-1}$ in TT format. 
  \State Compute the joint diagonalization of $\bm{S}^{(k)}$ and $\bm{M}^{(k)}$, 
  \begin{equation*}
  \begin{split}
      \left\{\begin{array}{cc}
           \bm{D}_k\quad \leftarrow\quad \bm{U}_k\bm{S}^{(k)}\bm{V}_k,\\
           \bm{\Lambda}_k\quad \leftarrow\quad \bm{U}_k\bm{M}^{(k)}\bm{V}_k,
      \end{array}\right.\quad k=1,2,\ldots,d,
  \end{split}
  \end{equation*}
  and denote the vectors consisting of the diagonal entries of $\bm{D}_k$ and $\bm{\Lambda}_k$ as $\bm{\mu}^{(k)}$ and $\bm{\lambda}^{(k)}$.
  \State Construct the TT representation \eqref{eq:tt-coefficient} of the tensor $\tensor{L}$.
  \State Compute the Hadamard inverse of $\tensor{L}$ by \Cref{algo:newton}: 
  $$\tensor{X}\quad \leftarrow\quad \tensor{L}^{\odot-1}.$$
  \State Represent the low-rank approximation of $\bm{L}^{-1}$ in TT format by \eqref{eq_final_inverse}:
  \begin{equation*}
      \begin{gathered}
          \bm{X}\ \leftarrow\  \texttt{Expanding}\left(\tensor{X}\right)\times_1\bm{U}_1^{-1}\times_2\bm{V}_{1}^{-1}\cdots\times_{2k-1}\bm{U}_{k}^{-1}\times_{2k}\bm{V}_{k}^{-1}\cdots\times_{2d-1}\bm{U}_d^{-1}\times_{2d}\bm{V}_d^{-1}.
      \end{gathered}
  \end{equation*}
	\end{algorithmic}
\end{algorithm}

\begin{remark}
    Step 4 of \Cref{algo:solver} is redundant in solving the linear system $\bm{L}\bm{u} = \bm{f}$, which can be replaced by 
    \[
    \tensor{U}\quad\leftarrow\quad\left(\tensor{L}^{\odot-1}\odot\hat{\tensor{F}}\right)\times_{1}(\bm{V}_1\bm{U}_1^{-1})\times_2(\bm{V}_2\bm{U}_2^{-1})\cdots\times_{d}(\bm{V}_d\bm{U}_d^{-1}),
    \]
where $\tensor{U}$ and $\hat{\tensor{F}}$ are the tensor corresponding to 
the solution $\bm{u}$ and 
$\left(\bm{U}_1\otimes\bm{U}_2\otimes\cdots\otimes\bm{U}_d\right)\bm{f}$, respectively. 
\end{remark}

\begin{remark}\label{rmk:accuracy}
    {We remark that the accuracy of the inverse matrix obtained by \Cref{algo:solver} depends on the accuracy of computing the Hadamard inverse of $\tensor{L}$, i.e., step 3 of \Cref{algo:solver}.} Specifically, we have 
\begin{equation}\label{eq:error_analysis}
\begin{array}{rl}
     \frac{\|\bm{X} - \bm{L}^{-1}\|_F}{\|\bm{L}^{-1}\|_F}
        &  =\frac{\|\texttt{Expanding}\left(\tensor{X}-\tensor{L}^{\odot-1}\right)\times_1\bm{U}_1^{-1}\times_2\bm{V}_{1}^{-1}\cdots\times_{2k-1}\bm{U}_{k}^{-1}\times_{2k}\bm{V}_{k}^{-1}\cdots\times_{2d-1}\bm{U}_d^{-1}\times_{2d}\bm{V}_d^{-1}\|_F}{\|\texttt{Expanding}\left(\tensor{L}^{\odot-1}\right)\times_1\bm{U}_1^{-1}\times_2\bm{V}_{1}^{-1}\cdots\times_{2k-1}\bm{U}_{k}^{-1}\times_{2k}\bm{V}_{k}^{-1}\cdots\times_{2d-1}\bm{U}_d^{-1}\times_{2d}\bm{V}_d^{-1}\|_F}\\
    &\leq \prod\limits_{k=1}^d\kappa\left(\bm{U}_k\right)\kappa\left(\bm{V}_k\right)\frac{\|\tensor{X}-\tensor{L}^{\odot-1}\|_F}{\|\tensor{L}^{\odot-1}\|_F}=\prod\limits_{k=1}^d\kappa\left(\bm{U}_k\right)\kappa\left(\bm{V}_k\right)\frac{\|\tensor{L}^{\odot-1}\odot\left(\tensor{L}\odot\tensor{X}-\tensor{E}\right)\|_F}{\|\tensor{L}^{\odot-1}\odot\tensor{E}\|_F}\\
    & \leq\prod\limits_{k=1}^d\kappa\left(\bm{U}_k\right)\kappa\left(\bm{V}_k\right)\kappa\frac{\|\tensor{L}\odot\tensor{X}-\tensor{E}\|_F}{\|\tensor{E}\|_F},
\end{array}
\end{equation}
where $\tensor{X}$ is the low-rank TT approximation of $\tensor{L}^{\odot-1}$ obtained by \Cref{algo:newton}, and $\kappa=\frac{\max(|\tensor{L}|)}{\min(|\tensor{L}|)}$.
\end{remark}

\subsection{Complexity analysis}

To facilitate our analysis of the time and memory costs of \Cref{algo:solver}, we assume that $n_k\approx n$ for all $k=1,2,\ldots,d$.
In step 1, the time cost of diagonalization depends on the structure of $\bm{S}^{(k)}$ and $\bm{M}^{(k)}$. For a general $n\times n$ diagonalizable matrix, it requires $\mathcal{O}(n^3)$ time cost. However,
most matrices discretized from PDE operators have Toeplitz structure such as the Laplace operator, which can be diagonalized by the fast Fourier transform (FFT) \cite{golub2013matrix}, it only requires $\mathcal{O}(n^2\log n)$ time cost. Further, when using the QTT format to represent the matrix, the time cost of FFT can be reduced to $\mathcal{O}(\log^2 n)$ \cite{dolgov2012superfast,chen2023quantum}, and we leave it for future work. Additionally, the memory cost in step 1 is $\mathcal{O}(dn^2 + dn)$. 
Since the TT representation of $\tensor{L}$ can be explicitly constructed via concatenation \cite{oseledets2011tensor,kazeev2012low}, there is no computational cost in step 2, but only a memory cost of $\mathcal{O}(dn)$.
The key of \Cref{algo:solver} is to compute the Hadamard inverse of $\tensor{L}$, i.e., step 3. Here we assume that the initial tensor chosen in \Cref{algo:newton} is close enough to $\tensor{L}^{\odot-1}$, and let $(1,r_1,\ldots,r_{d-1},1)$ be the TT rank of $\tensor{L}^{\odot-1}$ and $r = \max\{r_1,r_2,\ldots,r_{d-1}\}$.
As can be seen from \Cref{algo:newton}, its main computational cost in each iteration comes from the Hadamard product operation on TT formats and TT-rounding, whose time costs are $\mathcal{O}(dnr^4)$ and $\mathcal{O}(dnr^6)$, respectively. Correspondingly, the memory cost required in step 3 is $\mathcal{O}(dnr^4)$.
In step 4, the main computation is the TTMc operation in TT formats, whose time cost is bounded by $\mathcal{O}(dn^3r)$, and the memory cost is $\mathcal{O}(dnr + dn^2)$.

In summary, the overall time cost of \Cref{algo:solver} is no more than $\mathcal{O}(dn^3+(dnr^4+dnr^6)\times num\_iter+dn^3r)$, where $num\_iter$ is the number of iterations of Newton iteration. And the memory cost of \Cref{algo:solver} is $\mathcal{O}(dn^2+dn+dnr^4+dnr)$. It is worth mentioning that when the TT rank of $\tensor{L}^{\odot-1}$ is small enough, the time and memory costs grow linearly with the dimension $d$, which means that \Cref{algo:solver} can be efficiently used to solve large-scale and high-dimensional problems. In the next section, we will give a rigorous analysis to demonstrate the low-rank property of $\tensor{L}^{\odot-1}$. 
\section{Low-rank analysis in the TT format}\label{sec:analysis}

In this section, we give a rigorous analysis of the low-rank structure of the inverse matrix $\bm{L}^{-1}$ in the TT format, which is the key to the efficiency and application of the proposed TT-based matrix inversion method. More specifically, we give a computationally verifiable sufficient condition such that $\bm{L}^{-1}$ can be well approximated in a low-rank TT format, and further provides an estimate of the upper bound of its TT rank.

It is known that the TT rank of the tensor corresponding to the inverse matrix $\bm{L}^{-1}$ is equal to $\tensor{L}^{-1}$ (see \Cref{sec:step4} for details), whose $(j_1,j_2,\ldots,j_d)$th element of $\tensor{L}^{\odot-1}$ can be written as
\begin{equation}\label{eq:elementwise}
    \tensor{L}^{\odot-1}({j_1,j_2,\ldots,j_d}) = \frac{1}{\sum\limits_{k=1}^{d}\lambda_{j_1}^{(1)}\cdots\lambda_{j_{k-1}}^{(k-1)}\mu_{j_k}^{(k)}\lambda_{j_{k+1}}^{(k+1)}\cdots\lambda_{j_d}^{(d)}},
\end{equation}
where $\left\{\mu_{j_k}^{(k)},\lambda_{j_k}^{(k)}: j_k = 1, 2, \ldots, n_k\right\}$ is the same as \Cref{subsec:diag}. Therefore, we can analyze the low-rank structure of $\bm{L}^{-1}$ in the TT format by $\tensor{L}^{\odot-1}$ instead of $\bm{L}^{-1}$ itself. 
Before that, we present the definition of tensors with displacement structures \cite{shi}.

\begin{definition}\label{def:displacement_tensor}
    (See \cite{shi}.) Let $\tensor{X}\in\mathbb{C}^{n_1\times n_2\times \cdots \times n_d}$ be a $d$th-order complex tensor. We say that it has $\left(\bm{A}^{(1)},\bm{A}^{(2)},\ldots,\bm{A}^{(d)}\right)$-displacement structure of $\tensor{F}\in\mathbb{C}^{n_1\times n_2\times\cdots\times n_d}$ if it satisfies the multidimensional Sylvester equation:
    \begin{equation}\label{eq:displace_tensor}
        \tensor{X}\times_1\bm{A}^{(1)}+\tensor{X}\times_{2}\bm{A}^{(2)}+\cdots+\tensor{X}\times_{d}\bm{A}^{(d)} = \tensor{F},
    \end{equation}
    where $\bm{A}^{(k)}$ is an $n_k\times n_k$ complex matrix for all $k = 1, 2, \ldots, d$.
\end{definition}

According to the explicit form of the $(j_1,j_2,\ldots,j_d)$th element of $\tensor{L}^{\odot-1}$ \eqref{eq:elementwise}, we prove that the tensor $\tensor{L}^{\odot-1}$ has a displacement structure of a rank-one TT tensor, as stated in \Cref{thm:displacement}.

\begin{theorem}\label{thm:displacement}
If the matrices $\left\{\bm{M}^{(k)}:k=1,2,\ldots,d\right\}$ are invertible, then $\tensor{L}^{\odot-1}$ has $\left(\bm{A}^{(1)},\bm{A}^{(2)},\ldots,\bm{A}^{(d)}\right)$-displacement structure of $\tensor{F}=\bm{f}^{(1)}\otimes\bm{f}^{(2)}\otimes\cdots\otimes\bm{f}^{(d)}$, where 
\begin{equation*}
    \bm{A}^{(k)} = \text{diag}\left(\frac{\mu_{1}^{(k)}}{\lambda_1^{(k)}},\frac{\mu_{2}^{(k)}}{\lambda_2^{(k)}},\ldots,\frac{\mu_{n_k}^{(k)}}{\lambda_{n_k}^{(k)}}\right)\quad \text{and}\quad \bm{f}^{(k)}=\left(\frac{1}{\lambda_{1}^{(k)}},\frac{1}{\lambda_{2}^{(k)}},\ldots,\frac{1}{\lambda_{n_k}^{(k)}}\right)
\end{equation*}
for all $k=1,2,\ldots,d$.
\end{theorem}

\begin{proof}
    Since the matrices $\left\{\bm{M}^{(k)}:k=1,2,\ldots,d\right\}$ are invertible, i.e., $\lambda_{j_k}^{(k)}\neq0$ for all $j_k=1,2,\ldots,n_k$ and $k=1,2,\ldots,d$, the $(j_1,j_2,\ldots,j_d)$th element of $\tensor{L}^{\odot-1}$ can be rewritten as 
    \begin{equation}\label{eq:elementwise-1}
     \frac{1}{\prod\limits_{k=1}^d\lambda_{j_k}\left(\sum\limits_{k=1}^{d}\frac{\mu_{j_k}^{(k)}}{\lambda_{j_k}^{(k)}}\right)}.
\end{equation}
It satisfies 
\begin{equation*}
    \frac{\frac{\mu_{j_1}^{(1)}}{\lambda_{j_1}^{(1)}}}{\prod\limits_{k=1}^d\lambda_{j_k}^{(k)}\left(\sum\limits_{k=1}^{d}\frac{\mu_{j_k}^{(k)}}{\lambda_{j_k}^{(k)}}\right)} + \frac{\frac{\mu_{j_2}^{(2)}}{\lambda_{j_2}^{(2)}}}{\prod\limits_{k=1}^d\lambda_{j_k}^{(k)}\left(\sum\limits_{k=1}^{d}\frac{\mu_{j_k}^{(k)}}{\lambda_{j_k}^{(k)}}\right)} + \cdots + \frac{\frac{\mu_{j_d}^{(d)}}{\lambda_{j_d}^{(d)}}}{\prod\limits_{k=1}^d\lambda_{j_k}^{(k)}\left(\sum\limits_{k=1}^{d}\frac{\mu_{j_k}^{(k)}}{\lambda_{j_k}^{(k)}}\right)} = \frac{1}{\prod\limits_{k=1}^d\lambda_{j_k}^{(k)}},
\end{equation*}
which implies that the following equation 
\begin{equation*}
    \tensor{L}^{\odot-1}\times_{1}\bm{A}^{(1)}+\tensor{L}^{\odot-1}\times_2\bm{A}^{(2)}+\cdots+\tensor{L}^{\odot-1}\times_{d}\bm{A}^{(d)} = \tensor{F}
\end{equation*}
holds.
\end{proof}

Next, we will show how to estimate the distribution of singular values of the unfolding matrix for $\tensor{L}^{\odot-1}$ when $\tensor{L}^{\odot-1}$ has a displacement structure. Two key lemmas (Lemma \ref{lem:displacement_tensor} and \ref{lem:displacement-2}), which are found in \cite{beckermann2017singular,beckermann2019bounds,shi}, are given to establish our main theorem (Theorem \ref{thm:inverse_displacement}). We state the two lemmas as follows.

\begin{lemma}\label{lem:displacement_tensor}
   (See \cite{shi}.) Let $\tensor{X}\in\mathbb{C}^{n_1\times n_2\times\cdots\times n_d}$ be a $d$th-order tensor with $\left(\bm{A}^{(1)},\bm{A}^{(2)},\ldots,\bm{A}^{(d)}\right)$-displacement structure of $\tensor{F}\in\mathbb{C}^{n_1\times n_2\times\cdots\times n_d}$, and $\left\{\bm{A}^{(k)}\in\mathbb{C}^{n_k\times n_k}:k=1,2,\ldots,d\right\}$ be normal matrices.
 Denote the eigenvalues of $\bm{A}^{(k)}$ as $\lambda\left(\bm{A}^{(k)}\right)$, and let $\mathcal{E}_{k}$ and $\mathcal{F}_{k}$ be complex sets such that 
 \begin{equation*}
     \lambda\left(\bm{A}^{(1)}\right)+\cdots+\lambda\left(\bm{A}^{(k)}\right)\subset\mathcal{E}_{k}\quad \text{and}\quad -\left(\lambda\left(\bm{A}^{(k+1)}\right)+\cdots+\lambda\left(\bm{A}^{(d)}\right)\right)\subset\mathcal{F}_{k}.
 \end{equation*}
Then for each $k\in\{1,2,\ldots,d-1\}$, the singular values of $\bm{X}_{(\mathcal{I}_k)}$ satisfy
    \begin{equation}\label{eq:sv_tensor}
        \sigma_{i+\nu_kj}\left(\bm{X}_{(\mathcal{I}_k)}\right)\leq Z_{j}\left(\mathcal{E}_k,\mathcal{F}_k\right)\sigma_{i}\left(\bm{X}_{(\mathcal{I}_k)}\right),\quad 1\leq i+\nu_kj\leq\min\left\{\prod\limits_{s=1}^kn_s,\prod\limits_{s=k+1}^dn_s\right\},
    \end{equation}
    where $(1,\nu_{1},\ldots,\nu_{d-1},1)$ is the TT rank of $\tensor{F}$, and $Z_j(\mathcal{E}_k,\mathcal{F}_k)$ is the Zolotarev number with respect to $\mathcal{E}_k$ and $\mathcal{F}_k$.
\end{lemma}


From Lemma \ref{lem:displacement_tensor}, we know that the relationship between two singular values of $\bm{X}_{(\mathcal{I}_k)}$. The tensor $\tensor{X}$ has a low-rank tensor-train approximation if the singular values of the unfolding matrices decrease rapidly. This depends on the TT rank of $\tensor{F}$ and the Zolotarev number with respect to two complex sets associated with eigenvalues of $\bm{A}_k$. 
For the tensor $\tensor{L}^{\odot-1}$, \Cref{thm:displacement} shows that it has a displacement structure of a rank-one TT tensor, so we only need to estimate the bound of the Zolotarev number. We first present the definition of the Zolotarev number, as stated in Definition \ref{def:zoloterev}.

\begin{definition}\label{def:zoloterev}
   (See \cite{beckermann2017singular,beckermann2019bounds}.) Let $\mathcal{E}$ and $\mathcal{F}$ be two sets of complex numbers, then the corresponding Zolotarev number is defined as
   \begin{equation*}
    Z_j(\mathcal{E},\mathcal{F})=\inf\limits_{r\in\mathcal{R}_{j,j}}\frac{\sup\limits_{z\in\mathcal{E}}|r(z)|}{\inf\limits_{z\in\mathcal{F}}|r(z)|},
\end{equation*}
where $\mathcal{R}_{j,j}$ is the set of irreducible rational functions of the form $\frac{p(x)}{q(x)}$, and $p(x), q(x)$ are polynomials of degree at most $j$. 
\end{definition}

Here, we present a useful result in Lemma \ref{lem:displacement-2} for the evaluation of the Zolotarev number. It is adequate for our purpose of ensuring low-rank approximation.


\begin{lemma}\label{lem:displacement-2}
    (See \cite{beckermann2017singular,beckermann2019bounds}.) Let $\mathcal{E}$ and $\mathcal{F}$ be two sets in $\mathbb{C}$, and there exists two positive numbers $0\leq D_1\leq D_2$ such that 
    \begin{equation*}
        \mathcal{E}\subset\{z:|z-c|\leq D_1\}\quad \text{and}\quad \mathcal{F}\subset\{z:|z-c|\geq D_2\},
    \end{equation*}
    then we have
    \begin{equation}\label{eq:zolotarev-1}
        Z_j(\mathcal{E},\mathcal{F})=\left(Z_1(\mathcal{E},\mathcal{F})\right)^j,
    \end{equation}
     and $Z_1(\mathcal{E},\mathcal{F})=\frac{D_1}{D_2}$. 
\end{lemma}

According to \Cref{thm:displacement}, Lemma \ref{lem:displacement_tensor} and \ref{lem:displacement-2}, we can give a sufficient condition for the singular values of the unfolding matrices of $\tensor{L}^{\odot-1}$ to decay exponentially, i.e., the tensor $\tensor{L}^{\odot-1}$ can be well approximated by a low-rank TT tensor. A rigorous statement is given in \Cref{thm:inverse_displacement}.

\begin{theorem}\label{thm:inverse_displacement}
Let $\bm{X}_{(\mathcal{I}_k)}\in\mathbb{C}^{\prod\limits_{s=1}^{k}n_s\times\prod\limits_{s=k+1}^dn_s}$ be the matricization of $\tensor{L}^{\odot-1}$ along the index set $\mathcal{I}_k=\{1,\ldots,k\}$. For each $k\in\{1,2,\ldots,d-1\}$, if there exist a closed disk $\mathcal{D}_{k}=\{z:|z-c|\leq D\}$ in $\mathbb{C}$, satisfying
\begin{equation}\label{eq:cond-1}
\left\{\begin{array}{c}
     \left\{\sum\limits_{s=1}^k\frac{\mu_{j_s}^{(s)}}{\lambda_{j_s}^{(s)}}:j_s=1,2,\ldots,n_s; s=1,\ldots,k\right\}\subset\mathcal{D}_{k}  \\
     \left\{\sum\limits_{s=k+1}^d-\frac{\mu_{j_s}^{(s)}}{\lambda_{j_s}^{(s)}}:j_s=1,2,\ldots,n_s; s=k+1,\ldots,d\right\}\cap\mathcal{D}_{k}=\emptyset
\end{array}\right. 
\end{equation}
or
\begin{equation}\label{eq:cond-2}
\left\{\begin{array}{c}
     \left\{\sum\limits_{s=1}^k\frac{\mu_{j_s}^{(s)}}{\lambda_{j_s}^{(s)}}:j_s=1,2,\ldots,n_s; s=1,\ldots,k\right\}\cap\mathcal{D}_{k}=\emptyset  \\
      \left\{\sum\limits_{s=k+1}^d-\frac{\mu_{j_s}^{(s)}}{\lambda_{j_s}^{(s)}}:j_s=1,2,\ldots,n_s; s=k+1,\ldots,d\right\}\subset\mathcal{D}_{k}
\end{array}\right.
\end{equation}
then the singular values of $\bm{X}_{(\mathcal{I}_k)}$ decay exponentially. That is, there is a real number $q\in(0,1)$ such that 
\begin{equation}\label{eq:matricization_decay_1}
    \sigma_{j+1}\left(\bm{X}_{(\mathcal{I}_k)}\right)\leq q^{j}\sigma_{1}\left(\bm{X}_{(\mathcal{I}_k)}\right)
\end{equation}
holds for all $j=1,2,\ldots,\min\left\{\prod\limits_{s=1}^kn_s,\prod\limits_{s=k+1}^dn_s\right\}-1$.
\end{theorem}

\begin{proof}
Without loss of generality, we consider the case that condition \eqref{eq:cond-1} is satisfied, and the case where condition \eqref{eq:cond-2} is satisfied can be proved similarly. Since 
\[
\left\{\sum\limits_{s=k+1}^d-\frac{\mu_{j_s}^{(s)}}{\lambda_{j_s}^{(s)}}:j_s=1,2,\ldots,n_s; s=k+1,\ldots,d\right\}\cap\mathcal{D}_k=\emptyset,
\]
their distance is greater than 0 and denoted as $D'$. Let $\eta < D'$ be a sufficiently small positive real number and $$\mathcal{D}_k^{'}=\{z:|z-c|>D+D'-\eta\}.$$
It is clear that 
\[
\left\{\sum\limits_{s=k+1}^d-\frac{\mu_{j_s}^{(s)}}{\lambda_{j_s}^{(s)}}:j_s=1,2,\ldots,n_s; s=k+1,\ldots,d\right\}\subset\mathcal{D}^{'}_k.
\]
Then by \Cref{thm:displacement} and Lemma \ref{lem:displacement_tensor}, we know that the singular values of $\bm{X}_{(\mathcal{I}_k)}$ satisfy
\begin{equation}\label{eq:matricization_decay_2}
    \sigma_{j+1}\left(\bm{X}_{(\mathcal{I}_k)}\right) \leq Z_1(\mathcal{D}_k,\mathcal{D}^{'}_k)\sigma_{j}\left(\bm{X}_{(\mathcal{I}_k)}\right) \leq \left(Z_1(\mathcal{D}_k,\mathcal{D}^{'}_k)\right)^j\sigma_{1}\left(\bm{X}_{(\mathcal{I}_k)}\right)
\end{equation}
for all $j=1,2,\ldots,\min\left\{\prod\limits_{s=1}^kn_s,\prod\limits_{s=k+1}^dn_s\right\}-1$.
Further, by Lemma \ref{lem:displacement-2}, we have 
\[Z_1(\mathcal{D}_k,\mathcal{D}^{'}_k)=\frac{D}{D+D'-\eta}.\]
Since $\eta$ is small enough, there exists a real number $q\in(0,1)$ such that the inequality 
\[\frac{D}{D+D'-\eta}\leq q<1\]
holds, i.e., \eqref{eq:matricization_decay_1} is true.
\end{proof}

From \Cref{thm:inverse_displacement}, we can further give an estimate of the upper bound of the TT rank of the tensor such that it approximates $\tensor{L}^{\odot-1}$ with a predetermined accuracy requirement $\varepsilon\in(0,1)$. This is stated as the following Corollary \ref{coro:TT rank}.

\begin{corollary}[Estimation of TT ranks]\label{coro:TT rank}
    Let $\varepsilon\in(0,1)$ be a predetermined accuracy requirement and $\mathcal{S}_{\tensor{Y}}$ be a set as follows: 
    \begin{equation*}
        \mathcal{S}_{\tensor{Y}} := \left\{\tensor{Y}:\|\tensor{L}^{\odot-1} - \tensor{Y}\|_F\leq\varepsilon\|\tensor{L}^{\odot-1}\|_F\right\}.
    \end{equation*}
    Then there exists a tensor $\tensor{Y}\in\mathcal{S}_{\tensor{Y}}$ such that its the TT rank of $\tensor{Y}$ satisfies
    \begin{equation*}
        r_{k} \leq\left\lceil\log_{\tau^2}\left(\frac{(1-\tau^2)\varepsilon^2}{d-1}+\tau^{2n}\right)\right\rceil,\quad k=1,2,\ldots,d-1,
    \end{equation*} 
    where $\tau=\frac{D}{D+D'}$ and $n=\min\left\{n_kr_{k-1},\prod\limits_{s=k+1}^dn_s\right\}$.
\end{corollary}

     

In practice, the computational verifiability of the condition given in \Cref{thm:inverse_displacement} is very important. 
Next, we will introduce how to verify whether the given condition is true. Without loss of generality, we take condition \eqref{eq:cond-1} as an example, which consists of two parts. The first one is to find a closed disk $\mathcal{D}_k$ in $\mathbb{C}$ such that 
\begin{equation*}
    \left\{\sum\limits_{s=1}^k\frac{\mu_{j_s}^{(s)}}{\lambda_{j_s}^{(s)}}:j_s=1,2,\ldots,n_s; s=1,\ldots,k\right\}\subset\mathcal{D}_{k}.
\end{equation*}
Essentially, it is equivalent to finding a closed disk that contains all elements of the $k$th-order tensor $\tensor{A}\in\mathbb{C}^{n_1\times n_2\times\cdots\times n_k}$ with entries
\begin{equation*}
    \tensor{A}({j_1,j_2,\ldots,j_k}) = \frac{\mu_{j_1}^{(1)}}{\lambda_{j_1}^{(1)}}+\frac{\mu_{j_2}^{(2)}}{\lambda_{j_2}^{(2)}}+\cdots+\frac{\mu_{j_k}^{(k)}}{\lambda_{j_k}^{(k)}}.
\end{equation*}
We denote $\tensor{A}_1, \tensor{A}_2$
as the real and imaginary parts of $\tensor{A}$, and let $\alpha_1,\beta_1$ and $\alpha_2,\beta_2$  be the smallest and largest elements of them respectively.
It is easy to prove that all elements of $\tensor{A}$ are in the disk 
$$\mathcal{D}_k = \{z:|z-c|\leq D\},$$ 
where 
\begin{equation*}
   c = \frac{\alpha_1+\beta_1}{2}+i\frac{\alpha_2+\beta_2}{2}\quad \text{and} \quad D = \sqrt{\left(\frac{\beta_1-\alpha_1}{2}\right)^2+\left(\frac{\beta_2-\alpha_2}{2}\right)^2}.
\end{equation*}
The second one is to determine whether $\sum\limits_{s=k+1}^{d}-\frac{\mu_{j_s}^{(s)}}{\lambda_{j_s}^{(s)}}$ is in the disk $\mathcal{D}_k$. Similarly, it is equivalent to finding the smallest magnitude element of the $(d-k)$th-order tensor $\tensor{B}$ with entries 
\begin{equation*}
    \tensor{B}({j_{k+1},j_{k+2},\ldots,j_d}) = \frac{\mu_{j_{k+1}}^{(k+1)}}{\lambda_{j_{k+1}}^{(k+1)}}+\frac{\mu_{j_{k+2}}^{(k+2)}}{\lambda_{j_{k+2}}^{(k+2)}}+\cdots+\frac{\mu_{j_{d}}^{(d)}}{\lambda_{j_d}^{(d)}} + c,
\end{equation*}
that is, the smallest element of $\tensor{C}=\tensor{B}\odot\overline{\tensor{B}}$. If $\min(\tensor{C})>D^2$, then 
\begin{equation*}
    \left\{\sum\limits_{s=k+1}^d-\frac{\mu_{j_s}^{(s)}}{\lambda_{j_s}^{(s)}}:j_s=1,2,\ldots,n_s; s=k+1,\ldots,d\right\}\cap\mathcal{D}_{k}=\emptyset.
\end{equation*}
Otherwise, it does not hold. 

In conclusion, the verification of condition \eqref{eq:cond-1} can be transformed into finding the smallest and largest elements of tensors $\tensor{A}_1$, $\tensor{A}_2$, and $\tensor{C}$. Referring to the TT representation of $\tensor{L}$, we know that the tensors $\tensor{A}$ and $\tensor{B}$ can be represented in TT formats with rank-$(1,2,\ldots,2,1)$ and $(1,3,\ldots,3,1)$, respectively. Further, by the definitions of $\tensor{A}_1$, $\tensor{A}_2$, and $\tensor{C}$, they also have a low-rank TT representation. More specifically, the TT rank of $\tensor{A}_1$ and $\tensor{A}_2$ is at most $(1,4,\ldots,4,1)$, and the TT rank of $\tensor{C}$ is at most $(1,9,\ldots,9,1)$. Using the low-rank TT representation of $\tensor{A}_1$, $\tensor{A}_2$, and $\tensor{C}$, their smallest and largest elements can be efficiently obtained by methods proposed in \cite{espig2020iterative,sozykin2022ttopt}.

\section{Application to Numerical PDEs}\label{sec:application}
In this section, we will show that matrices derived from PDE operators involved in the Poisson, Boltzmann, and Fokker-Planck equations discretized by FDM not only have the form \eqref{eq:coefficient matrix}, but also satisfy the condition given in \Cref{thm:inverse_displacement}. 
More interestingly, it provides us with a new perspective to understand the effects of the mesh size and the time step size on solving these PDEs, that is, the perspective of the TT rank of their corresponding inverse matrices.

\subsection{Poisson equation}
The first example we consider is the $d$-dimensional Poisson in a hypercube with the Dirichlet boundary condition:
\begin{equation}\label{eq:poisson}
    -\Delta u(\bm{x}) = f(\bm{x}).
\end{equation}
We discrete the Poisson equation \eqref{eq:poisson} using the central difference scheme with a three-point stencil in each dimension, and then its corresponding linear system can be written as
\begin{equation}\label{eq:discrete-poisson}
    \bm{L}\bm{u} = \bm{f}\quad \text{with}\quad \bm{L} = \sum\limits_{k=1}^d\bm{I}_n\otimes\cdots\otimes\bm{I}_n\otimes(-\bm{\Delta})\otimes\bm{I}_n\otimes\cdots\otimes\bm{I}_n
\end{equation}
where 
\begin{equation*}
    -\bm{\Delta} = \frac{1}{h^2}\left[\begin{array}{ccccc}
       2  & -1 &  &  &  \\
        -1 & 2 & \ddots &  &  \\
         & \ddots & \ddots & -1 &  \\
         &  & -1 & 2 & -1 \\
         &  &  & -1 & 2 \\
    \end{array}\right]\in\mathbb{R}^{n\times n}
\end{equation*}
is a symmetric positive definite matrix that corresponds to the one-dimensional negative Laplace operator, and $n$ is the number of grids in each dimension. 

\begin{theorem}\label{thm:poisson}
    Let the matrix $\bm{L}$ be the form in \eqref{eq:discrete-poisson}, and $\tensor{L}\in\mathbb{R}^{n\times n\times\cdots\times n}$ be the $d$th-order tensor corresponding to $\bm{L}$ as defined in \eqref{eq_TTtensor_after_diag}. Then $\bm{L}$ is of the form \eqref{eq:coefficient matrix}, and the singular values of the matricization of $\tensor{L}^{\odot-1}$ decay exponentially.
\end{theorem}

\begin{proof}
    Since $-\bm{\Delta}$ is a real symmetric positive definite matrix, there exists an orthogonal matrix $\bm{U}\in\mathbb{R}^{n\times n}$ such that 
    $$\bm{U}^T\left(-\bm{\Delta}\right)\bm{U}=\bm{D},$$
    where $\bm{D} = \text{diag}\left(\mu_1,\mu_2,\ldots,\mu_n\right)$ is a diagonal matrix, and satisfies 
    $$\mu_1\geq\mu_2\geq\cdots\geq\mu_n>0.$$
    It is clear that the matrix $\bm{L}$ is of the form \eqref{eq:coefficient matrix}, and its corresponding tensor $\tensor{L}$ defined in \eqref{eq_TTtensor_after_diag} with entries
    \begin{equation*}
        \tensor{L}({j_1,j_2,\ldots,j_d}) = \sum\limits_{k=1}^d\mu_{j_k}.
    \end{equation*}
    Then by \Cref{thm:displacement}, we know that $\tensor{L}^{\odot-1}$ has $\left(\bm{A}^{(1)},\bm{A}^{(2)},\ldots,\bm{A}^{(d)}\right)$-displacement structure of $\bm{e}\otimes \bm{e}\otimes\cdots\otimes \bm{e}$, where 
    \begin{equation*}
        \bm{A}^{(k)} = \text{diag}\left(\mu_1,\mu_2,\ldots,\mu_n\right),\quad k=1,2,\cdots,d,
    \end{equation*}
    and $\bm{e}\in\mathbb{R}^n$ is a vector whose elements are all 1. For each $k\in\{1,2,\ldots,d-1\}$, let 
    \begin{equation*}
        c = \frac{k\mu_1+k\mu_n}{2}\quad\text{and}\quad D = \frac{k\mu_1-k\mu_n}{2},
    \end{equation*}
    and $\mathcal{D}=\{z:|z-c|\leq D\}$ be a closed disk,
    it is clear that the condition 
    \begin{equation*}
        \left\{\begin{array}{cc}
            \left\{\sum\limits_{s=1}^k\mu_{j_s}:j_s=1,2,\ldots,n;s=1,\ldots,k\right\}\subset\mathcal{D}  \\
             \left\{\sum\limits_{s=k+1}^d-\mu_{j_s}:j_s=1,2,\ldots,n;s=k+1,\ldots,d\right\}\cap\mathcal{D}=\emptyset
        \end{array}\right.
    \end{equation*}
    is met. \Cref{thm:inverse_displacement} illustrates that the singular values of the matricization of $\tensor{L}^{\odot-1}$ decay exponentially, and its corresponding decay factor is 
    $$\frac{k\mu_1-k\mu_n}{k\mu_1+(2d-k)\mu_n} = \frac{k\mu_1/\mu_n-k}{k\mu_1/\mu_n+2d-k},$$
    where $\mu_1/\mu_n$ is the condition number of the matrix $-\bm{\Delta}$.
\end{proof}

\subsection{Boltzmann-BGK equation}
In the second example, we consider the Boltzmann equation in the Bhatnagar-Gross-Krook (BGK) approximation, i.e., the Boltzmann-BGK equation, with the periodic boundary condition in the standard hypercube: 
\begin{equation}\label{eq:2d-bgk}
    \begin{gathered}
        \frac{f(\bm{x},\bm{v},t)}{\partial t}+\bm{v}\cdot\nabla_{\bm{x}}f(\bm{x},\bm{v},t)=\frac{\nu(\bm{x},t)}{\text{Kn}}[f_{\text{eq}}(\bm{x},\bm{v},t)-f(\bm{x},\bm{v},t)],\ (\bm{x},\bm{v})\in[-\pi,\pi]^d\times[-\pi,\pi]^d,
    \end{gathered}
\end{equation}
where $\text{Kn}$ is the Knudsen number, and $f(\bm{x},\bm{v},t)$ represents an unknown probability density function (PDF), which estimates the number of gas particles with velocity $\bm{v}\in\mathbb{R}^d$ at position $\bm{x}\in\mathbb{R}^d$ at time $t\in\mathbb{R}^+$.
The equilibrium PDF has the following form:
\begin{equation*}
    f_{\text{eq}}(\bm{x},\bm{v},t) = \frac{\rho(\bm{x},t)}{(2\pi T(\bm{x},t)/\text{Bo})^{d/2}}\exp\left(-\text{Bo}\frac{\|\bm{v}-\bm{U}(\bm{x},t)\|_2^2}{2T(\bm{x},t)}\right),
\end{equation*}
where $\text{Bo}$ is the Boltzmann number, and the density $\rho(\bm{x},t)$, mean velocity $\bm{U}(\bm{x},t)$, and temperature $T(\bm{x},t)$ are defined as 
\begin{equation*}
    \begin{gathered}
        \rho(\bm{x},t) = \int f(\bm{x},\bm{v},t)d\bm{v},\quad \bm{U}(\bm{x},t) = \frac{1}{\rho(\bm{x},t)}\int \bm{v}f(\bm{x},\bm{v},t)d\bm{v},\\
        T(\bm{x},t) = \frac{\text{Bo}}{d\rho(\bm{x},t)}\int \|\bm{v}-\bm{U}(\bm{x},t)\|_2^2f(\bm{x},\bm{v},t)d\bm{v}.
    \end{gathered}
\end{equation*}
The collision frequency $\nu(\bm{x},t)$ is usually set to
\begin{equation*}
    \nu(\bm{x},t) = \rho(\bm{x},t)KT(\bm{x},t)^{1-\mu},
\end{equation*}
where $K>0$ is a constant and the exponent $\mu$ of the viscosity law depends on the molecular interaction potential and the type of the gas.

We discrete the Boltzmann-BGK equation \eqref{eq:2d-bgk} using the Crank-Nicolson and central difference schemes in the time and spatial directions, respectively. Then in each time evolution, the corresponding linear system can be written as
\begin{equation}\label{eq:discrete-bgk}
    \left(\bm{I}_{n^d}+\frac{\delta t}{2}\bm{L}_{\bm{v}}\right)\bm{f}^{(n+1)}=\left(\bm{I}_{n^d}-\frac{\delta t}{2}\bm{L}_{\bm{v}}\right)\bm{f}^{(n)}+\delta t\frac{\bm{\nu}}{\text{Kn}}\left(\bm{f}_{\text{eq}}^{(n)}-\bm{f}^{(n)}\right),
\end{equation}
where $\delta t$ is the time step size. The matrix $\bm{L}_{\bm{v}}\in\mathbb{R}^{n^{d}\times n^{d}}$ is discretized from the differential operator $\bm{v}\cdot\nabla_{\bm{x}}$ by the central difference scheme, has the following form:
\begin{equation*}
    \bm{L}_{\bm{v}} = \sum\limits_{k=1}^d\bm{I}_{n^2}\otimes\cdots\otimes\bm{I}_{n^2}\otimes\left(\bm{V}^{(k)}\otimes\bm{\nabla}\right)\otimes\bm{I}_{n^2}\otimes\cdots\otimes\bm{I}_{n^2}
\end{equation*}
where $\bm{V}^{(k)}$ is a diagonal matrix whose entries consisting of $\bm{v}_k$ at grid points, and 
\begin{equation*}
    \bm{\nabla} = \frac{1}{h}\left[\begin{array}{ccccc}
       0  & 1 &  &  & -1 \\
        -1 & 0 & \ddots &  &  \\
         & \ddots & \ddots & 1 &  \\
         &  & -1 & 0 & 1 \\
        1 &  &  & -1 & 0 \\
    \end{array}\right]\in\mathbb{R}^{n\times n}
\end{equation*}
is an antisymmetric matrix.

\begin{theorem}\label{thm:boltzmann}
    Let the matrix $\bm{L}$ be the coefficient matrix of the linear system discretized from the Boltzmann-BGK equation, i.e., 
    \begin{equation*}
        \bm{L}=\sum\limits_{k=1}^d\bm{I}_{n^2}\otimes\cdots\otimes\bm{I}_{n^2}\otimes\left(\frac{1}{d}\bm{I}_{n^2}+\frac{\delta t}{2}\bm{V}^{(k)}\otimes\bm{\nabla}\right)\otimes\bm{I}_{n^2}\otimes\cdots\otimes\bm{I}_{n^2},
    \end{equation*}
    and $\tensor{L}\in\mathbb{R}^{n^2\times n^2\times\cdots\times n^2}$ be the $d$th-order tensor corresponding to $\bm{L}$ as defined in \eqref{eq_TTtensor_after_diag}. Then $\bm{L}$ is of the form \eqref{eq:coefficient matrix}, and the singular values of the matricization of $\tensor{L}^{\odot-1}$ decay exponentially when $\delta t/h$ is small enough.
\end{theorem}

\begin{proof}
    Since $\bm{\nabla}$ is an antisymmetric matrix, it satisfies
    \begin{equation}\label{eq:nabla}
        \bm{\nabla}^{T}\bm{\nabla} = \bm{\nabla}\bm{\nabla}^{T}
    \end{equation}
Then by \eqref{eq:nabla}, we have
\begin{small}
\begin{equation*}
    \left(\frac{1}{d}\bm{I}_{n^2}+\frac{\delta t}{2}\bm{V}^{(k)}\otimes\bm{\nabla}\right)^{T}\left(\frac{1}{d}\bm{I}_{n^2}+\frac{\delta t}{2}\bm{V}^{(k)}\otimes\bm{\nabla}\right) = \left(\frac{1}{d}\bm{I}_{n^2}+\frac{\delta t}{2}\bm{V}^{(k)}\otimes\bm{\nabla}\right)\left(\frac{1}{d}\bm{I}_{n^2}+\frac{\delta t}{2}\bm{V}^{(k)}\otimes\bm{\nabla}\right)^{T},
\end{equation*}
\end{small}
which means that 
$$\frac{1}{d}\bm{I}_{n^2}+\frac{\delta t}{2}\bm{V}^{(k)}\otimes\bm{\nabla} = \frac{1}{d}\bm{I}_{n^2}+\frac{\delta t}{2h}\bm{V}^{(k)}\otimes \left(h\bm{\nabla}\right)$$ 
is a normal matrix, and is diagonalizable. Clearly, the matrix $\bm{L}$ has the form \eqref{eq:coefficient matrix}. Let $\bm{U}\in\mathbb{C}^{n\times n}$ be a unitary matrix such that 
\begin{equation*}
    \bm{U}^{H}\left(\frac{1}{d}\bm{I}_{n^2}+\frac{\delta t}{2h}\bm{V}^{(k)}\otimes\left(h\bm{\nabla}\right)\right)\bm{U} = \frac{1}{d}+\frac{\delta t}{2h}\bm{D},
\end{equation*}
where $\bm{D}=\text{diag}\left(\mu_1,\mu_2,\ldots,\mu_{n^2}\right)$ is a diagonal matrix consisting of the eigenvalues of $\bm{V}^{(k)}\otimes\left(h\bm{\nabla}\right)$.
Then the corresponding tensor $\tensor{L}$ with entries
\begin{equation*}
    \tensor{L}({j_1,j_2,\ldots,j_d}) = 1 + \frac{\delta t}{2h}\sum\limits_{k=1}^d\mu_{j_k}.
\end{equation*}
Similarly, \Cref{thm:displacement} illustrates that $\tensor{L}^{\odot-1}$ has $\left(\bm{A}^{(1)},\bm{A}^{(2)},\ldots,\bm{A}^{(d)}\right)$-displacement structure of $\bm{e}\otimes\bm{e}\otimes\cdots\otimes\bm{e}$, where
\[
\bm{A}^{(k)}=\text{diag}\left(\frac{1}{d}+\frac{\delta t\mu_1}{2h},\frac{1}{d}+\frac{\delta t\mu_2}{2h},\ldots,\frac{1}{d}+\frac{\delta t\mu_{n^2}}{2h}\right),\quad k=1,2,\ldots,d,
\]
and $\bm{e}\in\mathbb{R}^{n^2}$ is a vector whose elements are all 1. Let $\mu_{\text{min},1},\ \mu_{\text{max},1}$ and $\mu_{\text{min},2},\ \mu_{\text{max},2}$ be the smallest and largest values of the real and imaginary parts of $\left\{\mu_1,\mu_2,\ldots,\mu_{n^2}\right\}$, we remark that they are bounded by a constant that only depends on the computational domain. Now we define a closed disk as $\mathcal{D}_{\delta t/h}=\{z:|z-c|\leq D\}$ with
\begin{equation*}
    c = \frac{k}{d}+\frac{k\delta t\mu_{\text{min},1}+k\delta t\mu_{\text{max},1}}{4h} + i\frac{k\delta t\mu_{\text{min},2}+k\delta t\mu_{\text{max},2}}{4h}
\end{equation*}
and
\begin{equation*}
    D = \sqrt{\left(\frac{k\delta t\mu_{\text{max},1}-k\delta t\mu_{\text{min},1}}{4h}\right)^2+\left(\frac{k\delta t\mu_{\text{max},2}-k\delta t\mu_{\text{min},2}}{4h}\right)^2}.
\end{equation*}
According to the definition of $\mathcal{D}_{\delta t/h}$, we know that it tends to the point $k/d$ when $\delta t/h$ tends to 0, and the set 
\[
\left\{\sum\limits_{s=k+1}^d-\left(\frac{1}{d}+\frac{\delta t}{2h}\mu_{j_s}\right):j_s=1,2,\ldots,n^2;s=k+1,\ldots,d\right\}
\]
tends to the point $(k-d)/d$. Clearly, the condition
\begin{equation*}
    \left\{\begin{array}{c}
         \left\{\sum\limits_{s=1}^k\frac{1}{d}+\frac{\delta t}{2h}\mu_{j_s}:j_s=1,2,\ldots,n^2;s=1,\ldots,k\right\} \subset\mathcal{D} \\
          \left\{\sum\limits_{s=k+1}^d-\left(\frac{1}{d}+\frac{\delta t}{2h}\mu_{j_s}\right):j_s=1,2,\ldots,n^2;s=k+1,\ldots,d\right\}\cap\mathcal{D} = \emptyset
    \end{array}\right.
\end{equation*}
is satisfied when $\delta t/h$ is small enough. Then by \Cref{thm:inverse_displacement}, we know that the singular values of the matricization of $\tensor{L}^{\odot-1}$ decay exponentially, and its corresponding decay factor approximates to 
$$\frac{k\delta t/h\sqrt{\left(\mu_{\text{max},1}-\mu_{\text{min},1}\right)^2+\left(\mu_{\text{max},2}-\mu_{\text{min},2}\right)^2}}{4+k\mu_{\text{max},1}\delta t/h+2k\mu_{\text{min},1}\delta t/h}.$$
\end{proof}

\subsection{Fokker-Planck equation}
The third example we consider is the $d$-dimensional Fokker-Planck equation with the open boundary condition, i.e., 
\begin{equation}\label{eq:d-fp}
\begin{gathered}
    \frac{\partial\rho(\bm{x},t)}{\partial t}=\nabla\cdot\left(\rho(\bm{x},t)\bm{A}\bm{x}\right)+\nabla\cdot\left(\nabla\cdot\left(\rho(\bm{x},t)\bm{D}\right)\right),\ (\bm{x}, t) \in \mathbb{R}^d \times \mathbb{R}^+,
\end{gathered}
\end{equation}
where $\rho(\bm{x},t)$ is the unknown probability density function (PDF) of $\bm{x}$ at time $t$ with the initial PDF $\rho_0(\bm{x})$. 
We refer to \cite{chertkov2021solution} to set $\bm{A} = \bm{I}_d$ and $\bm{D}=\frac{1}{2}\bm{I}_d$, which is a special case of Ornstein-Uhlenbeck process (OUP). Here we also use the Crank-Nicolson and central difference schemes to discrete the Fokker-Planck equation \eqref{eq:d-fp}. Then the corresponding linear system in each time evolution is as follows:
\begin{equation}\label{eq:discrete-fp}
    \left(\bm{I}_{n^d} - \frac{\delta t}{2}\bm{L}_{\bm{x}}\right)\bm{\rho}^{(n+1)} = \left(\bm{I}_{n^d}+\frac{\delta t}{2}\bm{L}_{\bm{x}}\right)\bm{\rho}^{(n)},
\end{equation}
where $\bm{L}_{\bm{x}}$ is the discrete form of $d+\bm{x}\cdot\nabla+\frac{1}{2}\Delta$. More specifically,
\begin{equation*}
    \bm{L}_{\bm{x}} = \sum\limits_{k=1}^d\bm{I}_n\otimes\cdots\otimes\bm{I}_n\otimes\bm{L}^{(k)}_{\bm{x}}\otimes\bm{I}_n\otimes\cdots\otimes\bm{I}_n
\end{equation*}
where $\bm{L}^{(k)}_{\bm{x}} = \bm{I}_n + \bm{X}^{(k)}\bm{\nabla}+\frac{1}{2}\bm{\Delta}$ and $\bm{X}^{(k)}$ is a diagonal matrix whose entries consisting of $\bm{x}_k$ at grid points. Since the open boundary condition can be approximated by the Dirichlet boundary condition on a hypercube \cite{risken1996fokker}, we have
\begin{equation*}\label{eq:adv-diffusion}
    \bm{\nabla} = \frac{1}{h}\left[\begin{array}{ccccc}
       0  & 1 &  &  &  \\
        -1 & 0 & \ddots &  &  \\
         & \ddots & \ddots & 1 &  \\
         &  & -1 & 0 & 1 \\
         &  &  & -1 & 0 \\
    \end{array}\right]\quad \text{and}\quad \bm{\Delta} =\frac{1}{h^2}\left[\begin{array}{ccccc}
       -2  & 1 &  &  &  \\
        1 & -2 & \ddots &  &  \\
         & \ddots & \ddots & -1 &  \\
         &  & 1 & -2 & 1 \\
         &  &  & 1 & -2 \\
    \end{array}\right],
\end{equation*} 
which are respectively antisymmetric and symmetric matrices.

\begin{theorem}\label{thm:fp}
    Let the matrix $\bm{L}$ be the coefficient matrix of the linear system discretized from the Fokker-Planck equation, i.e., 
   \begin{equation*}
    \bm{L} = \sum\limits_{k=1}^{d}\bm{I}\otimes\cdots\otimes\bm{I}\otimes\left(\left(\frac{1}{d} - \frac{\delta t}{2}\right)\bm{I} - \frac{\delta t}{2}\bm{X}^{(k)}\bm{\nabla}-\frac{\delta t}{4}\bm{\Delta}\right)\otimes\bm{I}\otimes\cdots\otimes\bm{I},
\end{equation*}
    and $\tensor{L}\in\mathbb{R}^{n\times n\times\cdots\times n}$ be the $d$th-order tensor corresponding to $\bm{L}$ as defined in \eqref{eq_TTtensor_after_diag}. Then the matrix $\bm{L}$ has the form \eqref{eq:coefficient matrix}, and the singular values of the matricization of $\tensor{L}^{\odot-1}$ decay exponentially when the mesh size $h$ and $\delta t/h$ is small enough.
\end{theorem}

\begin{proof}
The key to proving that $\bm{L}$ is of the form \eqref{eq:coefficient matrix} is to show that the matrix $\frac{1}{2}\bm{X}^{(k)}\bm{\nabla}+\frac{1}{4}\bm{\Delta}$ is diagonalizable. According to the explicit form of $\bm{X}^{(k)}$, $\bm{\nabla}$, and $\bm{\Delta}$, we can rewrite the matrix $\frac{1}{2}\bm{X}^{(k)}\bm{\nabla}+\frac{1}{4}\bm{\Delta}$ as
$\frac{1}{2h}\left(-\frac{1}{h}\bm{I}_n+\bm{A}\right)$,
where the matrix $\bm{A}\in\mathbb{R}^{n\times n}$ has the following form:
\begin{small}
\begin{equation*}
    \left[\begin{array}{ccccc}
        0 & x_0^{(k)}+h+\frac{1}{2h} & &  &  \\
        -x_0^{(k)}-2h+\frac{1}{2h} & 0 & \ddots &  & \\
         & \ddots & \ddots & x_0^{(k)}+(n-2)h+\frac{1}{2h} & \\
         & & -x_0^{(k)}-(n-1)h+\frac{1}{2h} & 0 & x_0^{(k)}+(n-1)h+\frac{1}{2h}\\
         & & & -x_0^{(k)}-nh+\frac{1}{2h} & 0 \\
    \end{array}\right],
\end{equation*}
\end{small}
and $x_0^{(k)}$ is the left endpoint of the computational domain in the $k$th dimension. Therefore, we only need to check whether $\bm{A}$ is diagonalizable. Since the set $\left\{x_0^{(k)}+ih:i=1,2,\ldots,n\right\}$ is bounded by a constant that does not depend on $n$ and $h$, both 
\begin{equation*}
    -x_0^{(k)}-ih+\frac{1}{2h}\quad\text{and}\quad x_0^{(k)}+(i-1)h+\frac{1}{2h}
\end{equation*}
are larger than 0 when the mesh size $h$ is small enough. Clearly, the following inequality
\begin{equation}\label{eq:fp-a}
    \left(-x_0^{(k)}-ih+\frac{1}{2h}\right)\cdot\left(x_0^{(k)}+(i-1)h+\frac{1}{2h}\right)>0
\end{equation}
holds for all $i=2,\ldots,n$. As referred in \cite{gantmacher2002oscillation} (pp. 67-69), the eigenvalues of $\bm{A}$ are real and distinct, which implies that $\bm{A}$ is diagonalizable.

Let $\bm{U}\in\mathbb{R}^{n\times n}$ be an invertible matrix such that 
$\bm{U}^{-1}\bm{A}\bm{U}=\bm{D}$, where $\bm{D} = \text{diag}\left(\mu_1,\mu_2,\ldots,\mu_n\right)$ is a 
diagonal matrix and satisfies
$$\mu_1>\mu_2>\cdots>\mu_n.$$
Then the corresponding tensor $\tensor{L}$ with entries 
\begin{equation*}
    \tensor{L}({j_1,j_2,\ldots,j_d}) =  1 - \frac{d\delta t}{2}-\frac{d\delta t}{2h^2} - \frac{\delta t}{2h}\sum\limits_{k=1}^d\mu_{j_k}.
\end{equation*}
Similar to the proof of \Cref{thm:boltzmann}, we can also prove that the singular values of the matricization of $\tensor{L}^{\odot-1}$ decay exponentially when $\delta t/h$ is small enough, and its corresponding decay factor is 
$$\frac{d(\mu_1-\mu_n)\delta t/h}{4d-2d\delta t-2d\delta t/h^2-(k\mu_1+(2d-k)\mu_n)\delta t/h}.$$
\end{proof}



\begin{remark}\label{rmk:discussion}
    The effects of the mesh size $h$ and the time step size $\delta t$ on solving PDEs can be characterized by the TT rank of their corresponding inverse matrices. 
\begin{itemize}
    \item For the Poisson equation, we know that the condition number of its corresponding coefficient matrix grows quadratically with $1/h$, and \Cref{thm:poisson} illustrates that the decay factor is positively correlated with the condition number. In other words, the TT rank of $\tensor{L}^{\odot-1}$ derived from the Poisson equation will increase as the mesh is refined, which results in higher computational costs. 
    \item For the Boltzmann and Fokker-Planck equations, 
    \Cref{thm:boltzmann} and \ref{thm:fp} illustrate that the corresponding decay factor is positively correlated with $\delta t/h$. Interestingly, the stability of discretization schemes can also be measured by $\delta t/h$, which implies that there is a potential connection between it and the TT rank of $\tensor{L}^{\odot-1}$ derived from the Boltzmann and Fokker-Planck equations.
\end{itemize}
\end{remark}
\section{Numerical experiments}\label{sec:example}

In this section, we demonstrate the advantages of the proposed TT-based matrix inversion method in solving PDEs analyzed in \Cref{sec:application} from the perspective of numerical experiments. For comparison, we also implement a matrix inversion method without TT representations, called the baseline method.
In the TT-based matrix inversion method, the implementations of the involved operations on TT formats are based on TT-Toolbox \cite{oseledets2019matlab}. If there is no special instruction, the error tolerances in \Cref{algo:newton} and TT-rounding are set to $1.0\times10^{-6}$ and $1.0\times10^{-8}$, respectively. To measure the accuracy of the TT-based matrix inversion method in solving PDEs, we define the relative error as follows:
\begin{equation*}
    \text{Relative error} = \frac{\|\bm{u}_{\text{tt}}-\bm{u}_{\text{ref}}\|_2}{\|\bm{u}_{\text{ref}}\|_2},
\end{equation*}
where $\bm{u}_{\text{tt}}$ is the solution obtained by the TT-based method, and $\bm{u}_{\text{ref}}$ is the reference solution.
All numerical experiments are conducted in MATLAB R2021b on a laptop equipped with an Intel Core i7-11390H CPU of 3.40 GHz.

\subsection{3D Poisson equation}

The first example is a 3D Poisson equation in the hypercube $[-1,1]^3$ with the Dirichlet boundary condition \cite{wang2024solving}:
\begin{equation}\label{eq:poisson-eq}
    \begin{gathered}
        -\Delta u(\bm{x}) = 6\pi^2\sum\limits_{k=1}^3\sin(2\pi x_k)\cdot\prod\limits_{i\neq k}^{3}\sin(\pi x_i),\ \bm{x}\in[-1,1]^3,\\
        u(\bm{x}) = 0,\ \bm{x}\in\partial[-1,1]^3,
    \end{gathered}
\end{equation}
whose exact solution is represented as
\[
u_*(\bm{x}) = \sum\limits_{k=1}^3\sin(2\pi x_k)\cdot\prod\limits_{i\neq k}\sin(\pi x_i).
\]
In this example, we set the number of grids in each dimension $n$ to $256$, $512$, and $1024$, respectively. \Cref{fig:poisson-1} illustrates the distribution of the singular values of the matricization of the tensor $\tensor{L}^{\odot-1}$ with different numbers of girds. From this figure, we observe that the singular values indeed decay exponentially, which means that the inverse of the coefficient matrix of the linear system discretized from the 3D Poisson equation \eqref{eq:poisson-eq} can be approximated well by a low-rank TT tensor. This is consistent with the theory in \Cref{thm:poisson}.

\begin{figure}[!htb]
    \centering
    \includegraphics[width=1.0\linewidth]{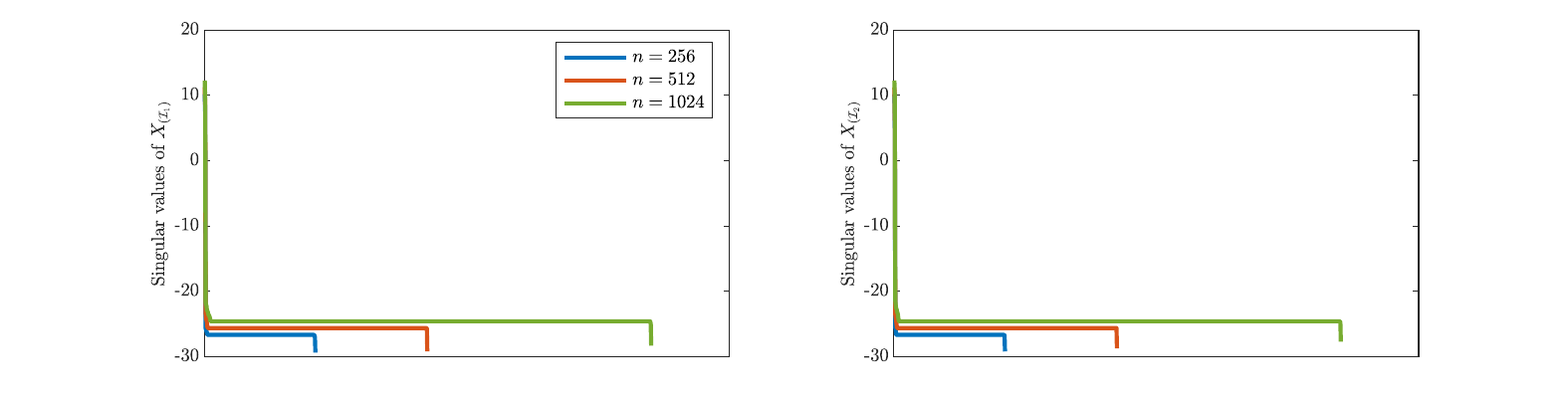}
    \caption{Distribution of the singular values of the matricization of the tensor $\tensor{L}^{\odot-1}$ with varying $n$ (the number of grids in each dimension).}
    \label{fig:poisson-1}
\end{figure}

\Cref{table:poisson} shows the comparison results of the baseline and TT-based matrix inversion methods in terms of computational efficiency and accuracy in solving the 3D Poisson equation \eqref{eq:poisson-eq}. 
We remark that the memory cost of the baseline method grows cubically with $n$, only the results with $n=256$ and $n=512$ are shown in \Cref{table:poisson} due to the limited memory. However, the proposed TT-based matrix inversion method can still obtain the solution corresponding to $n=1024$ with low time and memory costs, which illustrates its advantages in large-scale problems. In terms of accuracy, we can also see that the proposed TT-based matrix inversion method is able to achieve the same accuracy as the baseline method, and does not break the the convergence order of the central difference scheme used in this example. \Cref{fig:poisson-2} shows the solution obtained by the TT-based matrix inversion method is visually indistinguishable from the exact solution, which further demonstrates the accuracy of this method.
Moreover, we also observe that the TT rank of $\tensor{L}^{\odot-1}$ derived from the 3D Poisson equation grows as $n$ increases, which matches the theory in \Cref{thm:poisson}. It is worth mentioning that it reflects from another perspective that as the mesh is refined, solving the linear system corresponding to the Poisson equation will become more difficult due to the increase of the condition number.

\begin{table}[!htb]
\centering
\caption{Running times (unit: seconds), relative errors, and averaged TT ranks of the TT-based matrix inversion method for solving the 3D Poisson equation with varying $n$ (the number of grids in each dimension).}
\begin{tabular}{c|c|ccc}
\toprule
\multicolumn{2}{c|}{Algorithms}  & $n=256$ & $n=512$ & $n=1024$ \\
\midrule

\multicolumn{1}{c|}{\multirow{2}{*}{Baseline method}} & Running time & 1.00 & 10.20 & \multicolumn{1}{c}{\multirow{2}{*}{Out of memory}} \\
 & Relative error & $1.51\times10^{-4}$  & $3.77\times10^{-5}$ &  \\
\midrule         
\multicolumn{1}{c|}{\multirow{3}{*}{TT-based method}} &   Running time & 14.14 & 71.40 & 475.05  \\
 & Relative error   & $1.50\times10^{-4}$ & $3.37\times10^{-5}$ & $6.55\times10^{-6}$   \\
 & Averaged TT rank & 25 & 31 & 39 \\
\bottomrule
\end{tabular}
\label{table:poisson}
\end{table}

\begin{figure}[!htb]
    \centering
    \includegraphics[width=1.0\linewidth]{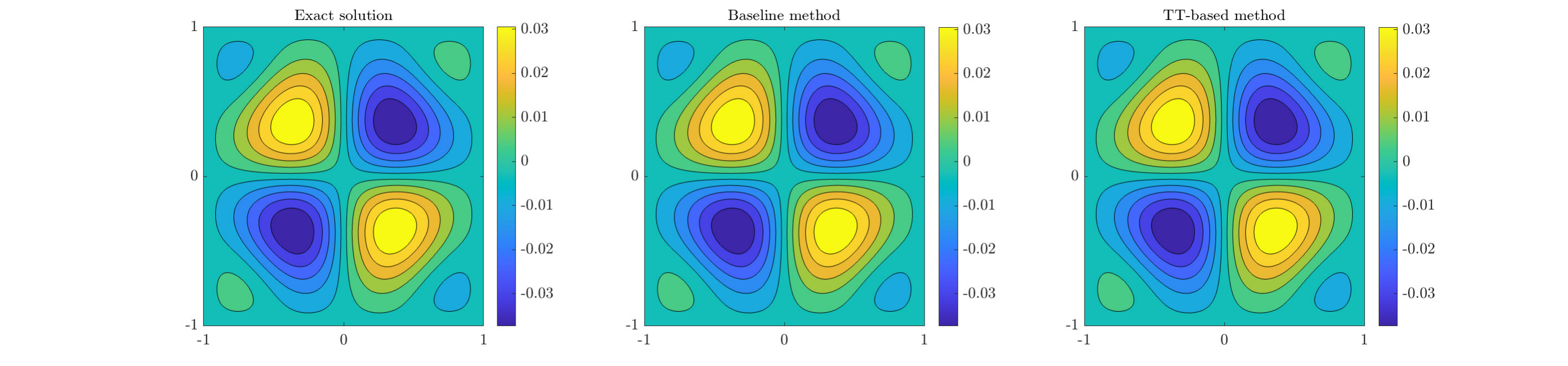}
    \caption{Solutions of the baseline and TT-based matrix inversion methods in the hyper-plane $(x_1,x_2,0)$ with $n=512$.}
    \label{fig:poisson-2}
\end{figure}

\subsection{2D2V Boltzmann-BGK equation}

The second example is a 2D2V Boltzmann-BGK equation from \cite{li2024solving} with the periodic boundary condition. We set the computational domain to $[-\pi,\pi]^2\times[-\pi,\pi]^2$, and the initial distribution $f(\bm{x},\bm{v},0)$ is taken as the Maxwellian distribution (i.e., the equilibrium distribution $f_{\text{eq}}(\bm{x},\bm{v},0)$) with macroscopic variables given by 
\begin{equation*}
    \rho(\bm{x},0) = 1+0.5\prod\limits_{k=1}^2\sin(x_k),\quad \bm{U}(\bm{x},0) = \bm{0},\quad \text{and}\quad T(\bm{x},0) = 1.
\end{equation*}
Other parameters of this example are presented in \Cref{table:boltzmann}. 

\begin{table}[!htb]
    \centering
    \caption{Parameter setting of the 2D2V Boltzmann-BGK equation.}
    \begin{tabular}{ccc}
      \toprule
      Variable & Value & Description \\
      \midrule
      $\mathrm{K}$ & 1.0 & Collision frequency pre-factor \\
      $\mu$ & 0.5 & Collision frequency temperature exponent \\
      $\mathrm{Kn}$ & 1 & Knudsen number \\
      $\mathrm{Bo}$ & 3.65 & Boltzmann number \\
      $\delta t$ & 0.00125, 0.0025, 0.005 & Time step size\\
      $n$ & 128 & Number of grids in each dimension \\
      \bottomrule
    \end{tabular}
    \label{table:boltzmann}
\end{table}

Since there is no exact solution in this example, we use the solution corresponding to $n=256$ as the reference solution.
For $t=1.0$s, the relative errors of the solutions obtained by the TT-based matrix inversion method with respect to $n$ are shown on the left side of \Cref{fig:boltzmann-1}. From it, we observe that the convergence order of the TT-based matrix inversion method is 2, which matches the second-order discrete scheme for the spatial direction. Also, the TT ranks shown on the right side of \Cref{fig:boltzmann-1} illustrate that the tensor $\tensor{L}^{\odot-1}$ derived from the 2D2V Boltzmann-BGK equation enjoys a low-rank structure in TT format. We remark that, the TT rank grows as the time step size increases, this is because the decay factor given in \Cref{thm:boltzmann} increases with respect to $\delta t/h$.

\begin{figure}[!htb]
    \centering
    \includegraphics[width=1.00\linewidth]{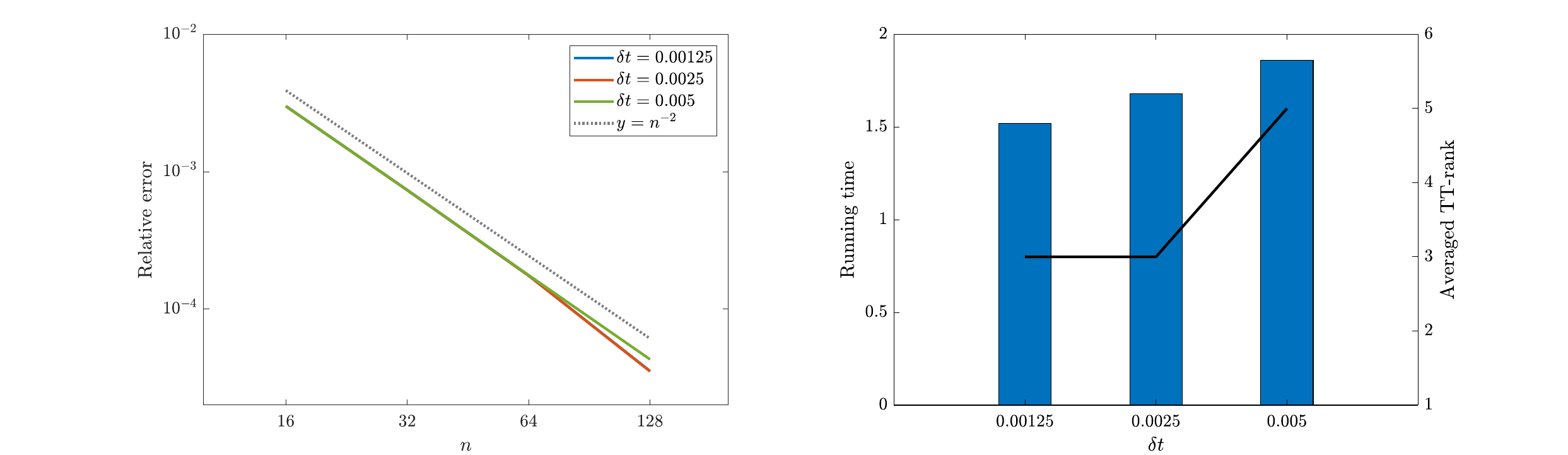}
    \caption{Numerical results of the TT-based matrix inversion method for solving the 2D2V Boltzmann-BGK equation with different time step sizes. Left: Relative errors with respect to $n$ at $t=1.0$s. Right: Running times (unit: seconds) and averaged TT ranks.}
    \label{fig:boltzmann-1}
\end{figure}

\begin{figure}[!htb]
    \centering
    \includegraphics[width=1.00\linewidth]{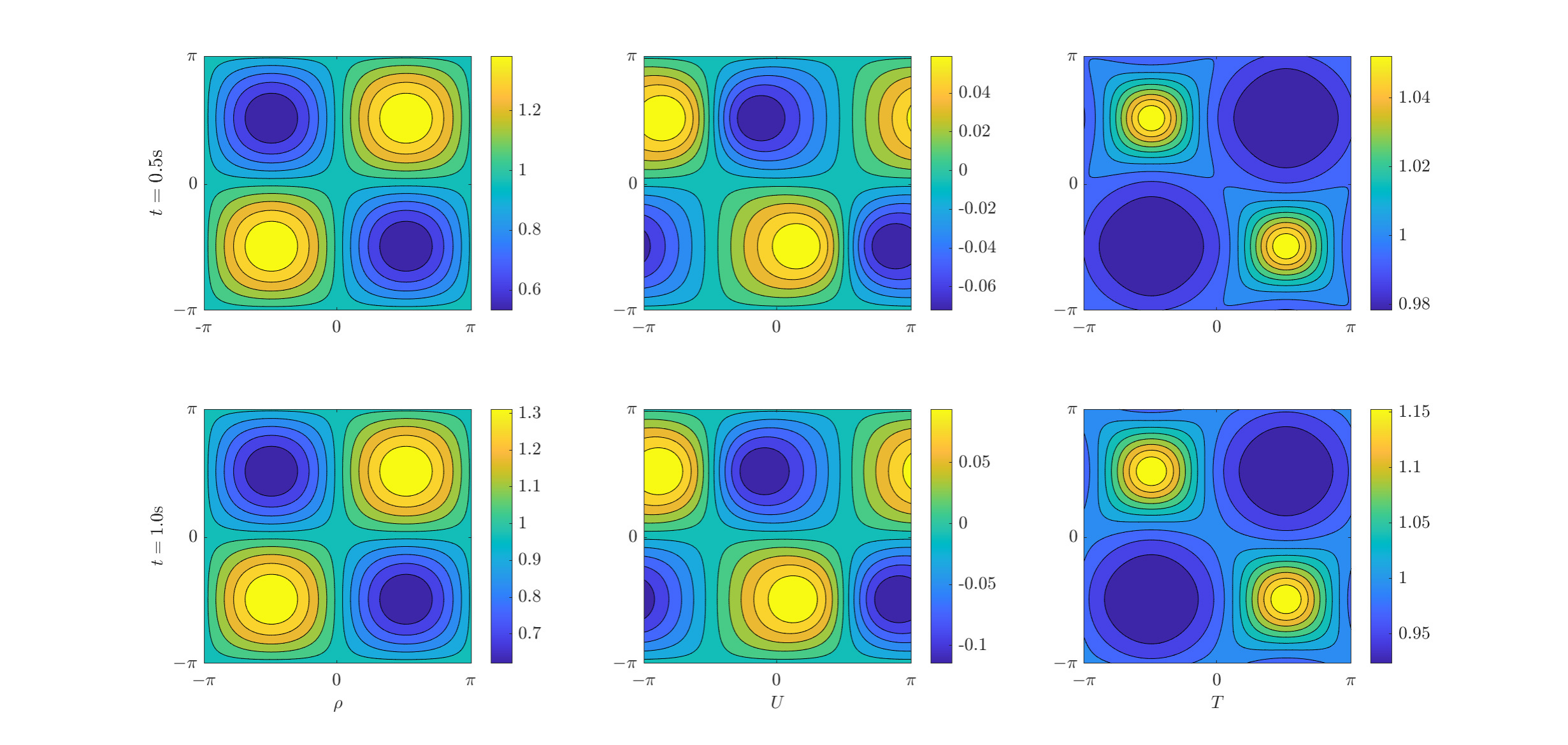}
    \caption{Macroscopic density $\rho(\bm{x},t)$, velocity $\bm{U}(\bm{x},t)$, and temperature $T(\bm{x},t)$ of the TT-based matrix inversion method with the time step size $\delta t=0.0025$ at $t=0.5$s and $1.0$s, respectively.}
    \label{fig:boltzmann-2}
\end{figure}

\subsection{Fokker-Planck equation}

The third example is to test the performance of the proposed TT-based matrix inversion method on the Fokker-Planck equation of different dimensions. Here the dimensions are taken as $3,\ 6,\ 9$, and $12$, and we set the initial distribution $\rho(\bm{x},0)$ to $\frac{1}{{(2\pi)^{d/2}}}e^{-\frac{\|\bm{x}\|^2_2}{2}}$, its corresponding exact solution can be represented as
\begin{equation}
    \rho(\bm{x},t) = \frac{1}{\left(\pi\sigma(t)\right)^{d/2}}e^{-\frac{\|\bm{x}\|_2^2}{\sigma(t)}}\quad \text{with}\quad {\sigma}(t) = {1+e^{-2t}}.
\end{equation}
Due to the open boundary condition of the Fokker-Planck equation, we refer to \cite{chertkov2021solution} to truncate the computational domain to $[-5,5]^d$, and the values of $\rho(\bm{x},t)$ on the boundary are set to 0. For the Fokker-Planck equation of different dimensions, the number of grids in each dimension $n$ and time step size $\delta t$ are set to $256$ and $0.00125$, respectively. To ensure the accuracy of the TT-based matrix inversion method in solving high-dimensional Fokker-Planck equations, we here set the error tolerances in \Cref{algo:newton} and TT-rounding to $1.0\times10^{-10}$ and $1.0\times10^{-12}$, respectively.

From \Cref{fig:fp-1}, we observe that for the Fokker-Planck equation of different dimensions, the relative errors obtained by the TT-based matrix inversion method do not accumulate with respect to $t$, so introducing TT representation will not break the stability of the used discrete scheme. 
Additionally, \Cref{fig:fp-1} shows that the tensor $\tensor{L}^{\odot-1}$ derived from the Fokker-Planck equation is well-matched to the TT format. Indeed, as referred in \Cref{sec:algorithm}, the time cost of the TT-based matrix inversion method increases linearly with the dimension $d$ due to the low-rank property of $\tensor{L}^{\odot-1}$ in TT format. \Cref{fig:fp-2} shows the solutions of the 3D Fokker-Planck equation obtained by the baseline and TT-based matrix inversion methods, which are visually indistinguishable from the exact solution. It further illustrates the correctness of the proposed TT-based matrix inversion method.

\begin{figure}[!htb]
    \centering
    \includegraphics[width=1.00\linewidth]{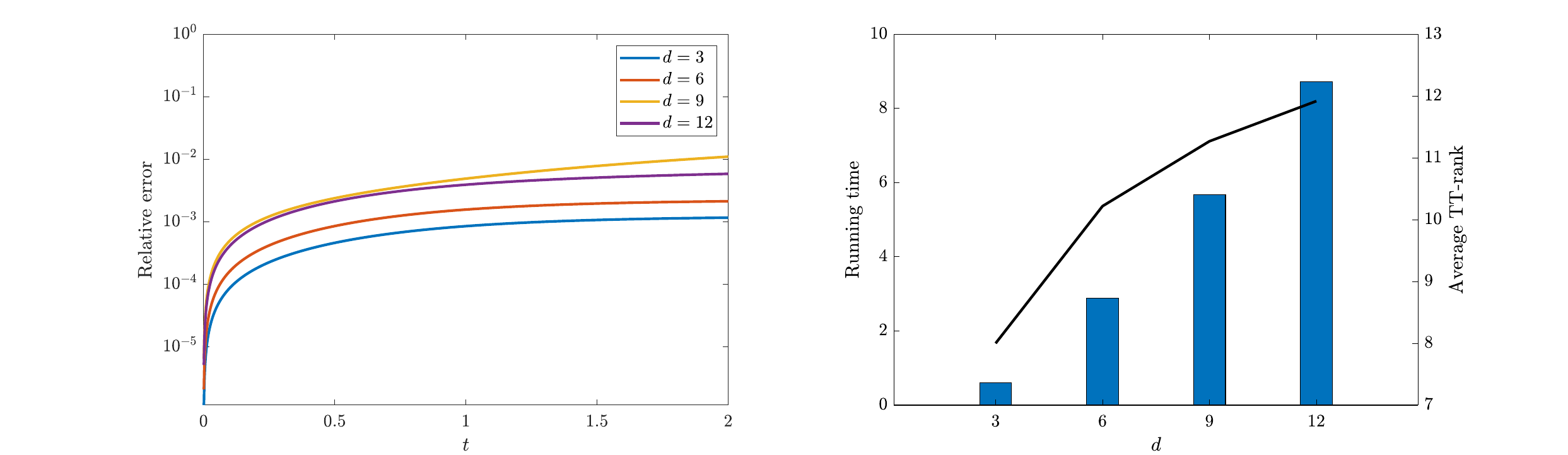}
    \caption{Numerical results of the TT-based matrix inversion method for solving the Fokker-Planck equation of different dimensions. Left: Relative errors. Right: Running times (unit: seconds) and averaged TT ranks.}
    \label{fig:fp-1}
\end{figure}

\begin{figure}[!htb]
    \centering
    \includegraphics[width=1.0\linewidth]{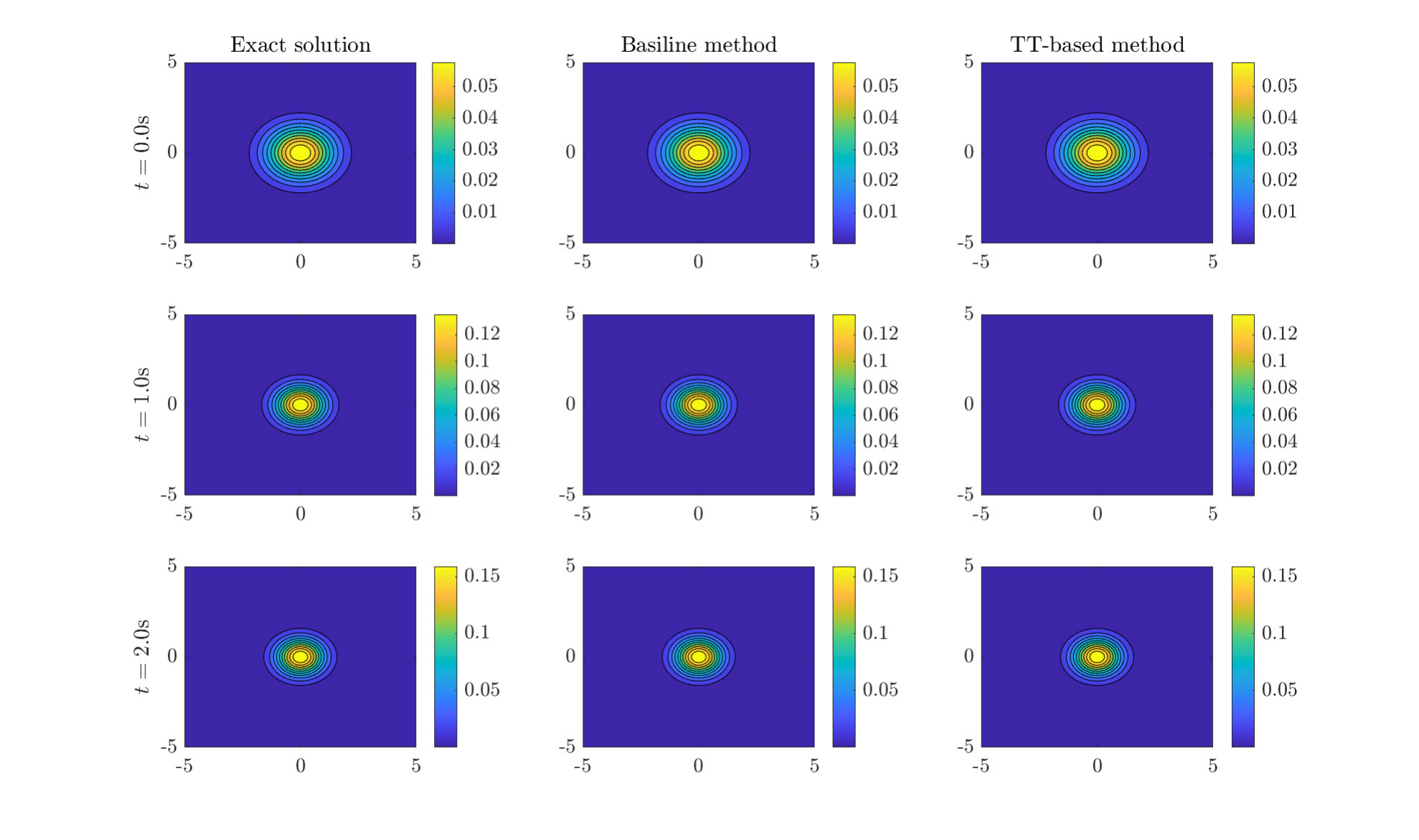}
    \caption{Solutions of the baseline and TT-based matrix inversion methods in the hyper-plane $(x_1,x_2,0)$ for the 3D Fokker-Planck equation.}
    \label{fig:fp-2}
\end{figure}
\section{Conclusions and future work}\label{sec:conclusion}
Although using TT-based iterative methods for solving large-scale linear systems related to numerical PDEs is a popular approach and often works well in practice, it needs a clear theoretical guarantee. In this paper, we conduct an in-depth analysis of the low-rank property for the inverse of a class of large-scale matrices in the TT format, which helps to take a step towards bridging the gap between theory and practical algorithms. Motivated by the analysis, we also propose a novel TT-based method for computing the inverse of matrices of the form \eqref{eq:coefficient matrix}. Our established theory answers the question of what conditions the structured matrix satisfies so that its inverse matrix is well-matched with the low-rank TT format. Based on the low-rank analysis, we prove that matrices discretized from the Poisson, Boltzmann, and Fokker-Planck equations are of the form \eqref{eq:coefficient matrix} and satisfy the given condition. Therefore, one can use the proposed TT-based matrix inversion method to numerically solve these PDEs with large degrees of freedom, which has also been validated by several numerical experiments. Possible future work may include: (1) applying our TT-based matrix inversion method to other problems such as inverse problems, preconditioners, etc; and (2) generalizing our low-rank analysis to more general structured matrices.\\

\noindent {\bf Acknowledgement.} C. Xiao has been supported by the Hunan Provincial Natural Science Foundation grant 2023JJ4005. K. Tang has been supported by the Hunan Provincial Natural Science Foundation grant 2024JJ6003.


\newpage
{\small

\bibliographystyle{amsalpha}
}

\end{document}